\documentclass[12pt,reqno]{amsart}
\usepackage{amsmath}
\usepackage{hyperref} 
\usepackage{amsfonts}
\usepackage{amssymb}
\usepackage{amsthm}
\usepackage{hyperref}
\usepackage{amssymb,amscd}
\usepackage{color}

\usepackage{paralist,csquotes,stackrel}
 
\usepackage[mathscr]{eucal}

\numberwithin{equation}{section}

\newtheorem{proposition}{\textbf{Proposition}}
\newtheorem{lemma}[proposition]{\textbf{Lemma}}

\newtheorem{theorem}[proposition]{\textbf{Theorem}}
\newtheorem{question}[proposition]{\textbf{Question}}

\theoremstyle{definition}
\newtheorem{definition}[proposition]{\textbf{Definition}}

\newtheorem{example}[proposition]{\textbf{Example}}
\newtheorem{remark}[proposition]{\textbf{Remark}}

\newcommand{\Lie}[1]{\operatorname{\textsl{#1}}}

\newcommand{\lie}[1]{\operatorname{\mathfrak{#1}}}

\newcommand{\sln}{\lie{sl}}

\newcommand{\SO}{\Lie{SO}}

\newcommand{\su}{\lie{su}}

\newcommand{\SL}{{\rm SL}}

\newcommand{\SU}{{\rm SU}}

\newcommand\C{{\mathbb C}}

\newcommand{\R}{{\mathbb R}}

\numberwithin{proposition}{section}

\begin{document}

\title[Real holomorphic sections of Deligne-Hitchin twistor space]{Real holomorphic 
sections of the Deligne-Hitchin twistor space}

\author[I. Biswas]{Indranil Biswas}

\address{School of Mathematics, Tata Institute of Fundamental
Research, Homi Bhabha Road, Mumbai 400005, India, and Mathematics Department, EISTI-University
Paris-Seine, Avenue du parc, 95000, Cergy-Pontoise, France}

\email{indranil@math.tifr.res.in}

\author[S. Heller]{Sebastian Heller}
\address{Fachbereich Mathematik,
Universit\"at Hamburg, 20146 Hamburg, Germany} 
\email{seb.heller@gmail.com}

\author[M. R\"oser]{Markus R\"oser}
\address{Institut f\"ur Differentialgeometrie, Welfengarten 1,
30167 Hannover, Germany} 
\email{roeser@math.uni-hannover.de}



\date{\today}

\begin{abstract}
We study the holomorphic sections of the Deligne-Hitchin moduli space of a
compact Riemann surface, especially the sections that are invariant under the natural
anti-holomorphic involutions of the moduli space. Their relationships with the harmonic maps
are established. As a by product, a question of Simpson on such sections, posed in
\cite{Si2}, is answered.
\end{abstract}
\maketitle

\tableofcontents

\section*{Introduction}

About thirty years ago, Hitchin introduced his famous self-duality equations on compact Riemann surfaces,
connecting a Hermitian metric $h$, a unitary connection $\nabla$ and a Higgs field $\Phi$,
as a 2-dimensional reduction of the Yang-Mills equation \cite{Hi1}. 
These self-duality equations are gauge invariant, and
the moduli space of their irreducible solutions, which is denoted
by $\mathcal M_{SD}$, is a smooth (finite-dimensional) manifold equipped with
rich geometric structures. In particular,
it admits a natural $L^2$-metric which is K\"ahler for two anti-commuting complex
structures, meaning it is a hyper-K\"ahler manifold. The first complex structure $I$ is
induced by a theorem in \cite{Hi1},
\cite{sim0} which identifies the irreducible solutions
of the self-duality equation, up to unitary gauge transformations,
with the stable
Higgs pairs up to complex gauge transformations; in one direction this map is
given by the forgetful map
\[(h,\,\nabla,\,\Phi)\,\longmapsto\,(\overline{\partial}^\nabla,\,\Phi)\, .\]
Since the moduli space of stable Higgs pairs has a natural complex structure
(this complex structure will be recalled in Section \ref{se1}) ---
the moduli space is in fact a smooth complex quasi-projective variety --- the
above forgetful map produces a
complex structure on $\mathcal M_{SD}$, which is denoted by $I$. The moduli space of Higgs bundles
is equipped with a holomorphic symplectic form and
an algebraically completely integrable system \cite{Hit2}.

The other complex structure $J$ on $\mathcal M_{SD}$
arises from the identification of $\mathcal M_{SD}$ with the moduli space of flat
complex connections due to a result of Donaldson, \cite{Do}, in the case of rank $n\,=\, 2$ and
by a result of Corlette in general \cite{Corlette}. The third complex structure $K$ is of
course $K = IJ = -JI$. The significance of the self-duality
equations in mathematics and theoretical physics is primarily due to, to quote
Hitchin, the various incarnations of the solutions, bringing together
geometry, topology, analysis and algebra in a harmonious way.

Every hyper-K\"ahler manifold $\mathcal M$ has a twistor space
associated to it. Topologically
it is just the product of the space $\mathcal M$ with the 2-sphere: 
\[
\{xI_p+yJ_p+zK_p
\,\mid\, p\,\in\,{\mathcal M},\, x^2+y^2+z^2\,=\,1\}\, .
\] 
It has a
tautological complex structure given by the natural complex structure on the
2-sphere and the above complex structure $xI_p+yJ_p+zK_p$ on ${\mathcal M}\times \{(x,y,z)\}$; the natural
projection of the twistor space to $\C P^1$ is evidently holomorphic. The importance of the twistor space construction
is due to the fact that the hyper-K\"ahler structure is encoded in holomorphic data on the twistor space.
In the case of the moduli space of solutions to the self-duality equations,
the corresponding twistor space, which is known as the Deligne-Hitchin moduli space and is
denoted by $\mathcal M_{DH}$, admits another interpretation which is due to Deligne; see \cite{Si}.
From Deligne's perspective, this Deligne-Hitchin moduli space $\mathcal M_{DH}$
is considered as the moduli space of $\lambda$-connections on the Riemann surface
$\Sigma$ glued
with the moduli space of $\lambda$-connections on the conjugate Riemann surface
$\overline\Sigma$ via the Riemann-Hilbert
isomorphism; for definitions and details see Section \ref{se1}. As a consequence, integrable 
surfaces like harmonic maps and minimal surfaces
can be described by sections of the twistor space $$\mathcal M_{DH}\,\longrightarrow\,\C P^1$$
satisfying certain reality conditions that depend on what the target space is. We note
that for the case of $G_\C\,=\,\SL(2,\C)$, these surfaces are lying in any of the following:
$S^3$, $H^3$ and its quotients, and their Lorentzian counterparts.

The main theme of this paper is to initiate investigations of real sections of the 
Deligne-Hitchin moduli space with an emphasis on their geometric meaning. There are 
two natural involutions of $\mathcal M_{DH}$, one covering the antipodal involution 
$\lambda\,\longmapsto\, -\overline{\lambda}^{-1}$
on $\C P^1$ (this does not have any fixed point) and the other covering $\lambda\,\longmapsto\, \overline{\lambda}^{-1}$
with the unit circle as fixed point set. For these two anti-holomorphic involutions, we obtain two 
classes of real sections by considering the two real forms of the complex Lie algebra $\mathfrak{sl}(2,\C).$ 
Under the extra condition that the real section is admissible (see Definition 
\ref{admissible}) those maps correspond to equivariant harmonic maps from the 
universal covering of the Riemann surface into the hyperbolic space, the space of 
oriented circles, the 3-sphere or the 3-dimensional anti-de Sitter space.

At this point we pause to describe the structure of the initial parts of the paper, before explaining
some of the results proved in the later parts. In Section \ref{se1}
we recall the basic facts 
about Hitchin's self-duality equations, the moduli space $\mathcal M_{SD}$ of 
solutions and its hyper-K\"ahler structure. We then describe the Deligne-Hitchin 
moduli space as the twistor space of $\mathcal M_{SD}$ from a differential geometry 
point of view. In Section \ref{sec2}, we derive some basic properties of sections 
which are used in the subsequent sections. In particular, we define the 
admissible sections (Definition 
\ref{admissible}) as those having a particular, simple form and giving rise to the 
most prominent examples of sections like the solutions of the self-duality 
equations. We also define the parity of a section and use this notion to prove that 
not all sections are admissible. In Section \ref{sec:real} we study real sections 
compatible with the two natural involutions $\rho$ and $\tau$ (Definition \ref{Def:sigmaReal}). We show that they come in two different flavors, 
a so-called negative one corresponding to the compact real form ${\rm SU}(2)$ of 
$\SL(2,\C)$ and a positive one corresponding to the split real form $\SL(2,\R)$ of 
$\SL(2,\C)$ (Definition \ref{Def:RealPosNeg}).

In Section \ref{sec:realm} we prove the existence of $\tau$-positive sections, which give counter-examples to a question 
of Simpson in \cite{Si2} (see \cite[\S~4, p.~235, Question]{Si}; reproduced
here in Question \ref{q1}); this 
is carried out in Proposition \ref{SDneg} and Theorem \ref{countertheorem}.

We prove that all $\tau$-negative sections which are admissible come from solutions of the 
self-duality equations, and give a geometric interpretation to the $\tau$-positive admissible 
sections as harmonic maps into the space of oriented circles in the 2-sphere. In Section 
\ref{sec:realp} we consider real sections compatible with the involution $\rho$ covering the involution 
$\lambda\,\longmapsto\,\overline\lambda^{-1}.$ Again, these $\rho$-real sections come in positive and negative 
flavor, where the negative ones are given by solutions of the harmonic map equations into the 
3-sphere, while the positive ones are related to harmonic maps into the anti-de Sitter space 
$AdS^3.$ The $\rho$-positive sections are equipped with a further integer invariant which turns 
out to be strongly related to the Toledo invariant for Higgs bundles in the Hitchin component. 
Here, and at other stages, we make use of the fix point set of the holomorphic involution $N$
covering $\lambda\,\longmapsto\,-\lambda$, in order to construct a correspondence
between the special 
subspaces of real sections for one real structure and the special subspaces of real sections for 
the other real structure. The most prominent example of this correspondence is certainly the 
well-known fact that minimal surfaces in $AdS^3$ are locally described by certain solutions of 
Hitchin's self-duality equations.

Finally, we also give a twistorial description of the 
hyper-symplectic structure on the moduli space of ``small'' solutions to the harmonic map 
equation for ${\rm SU}(2)$ in the spirit of \cite{HKLR}.

In recent years, there has been some interest in the automorphism groups of certain moduli spaces associated to
compact Riemann surfaces \cite{Ba}, \cite{BBS}, \cite{BiHe}.
We end the paper in Section \ref{auto} with an investigation of the
space of holomorphic automorphisms of the Deligne-Hitchin moduli space.
In particular, we prove that the 
holomorphic automorphisms of the Deligne-Hitchin moduli space, which are homotopic to the 
identity map, actually send the fibers of the twistor fibration to the fibers. 

\section{The Deligne-Hitchin moduli space ${\mathcal M}_{DH}$}\label{se1}

In this section we review the construction of the Deligne-Hitchin moduli space
${\mathcal M}_{DH}$ associated to a compact Riemann surface.

\subsection{The self-duality equations and their moduli space}\label{sec:SDaM}

Let $\Sigma$ be a compact connected Riemann surface of genus $g_\Sigma$, with $g_\Sigma \,\geq\, 2$.
The holomorphic cotangent bundle of $\Sigma$ will be denoted by $K$.

Given a
complex vector space $V$, we denote by $\underline{V}$ the trivial
$C^\infty$ complex vector bundle
$\Sigma\times V\,\longrightarrow\,\Sigma$ on $\Sigma$ with fiber $V$. Consider
$E\,=\,\underline{\C^2}$ equipped with
the standard Hermitian metric and complex volume form. Let $\nabla$ be an $\SU(2)$-connection on $E$, and
let $\Phi\,\in\, \Omega^{1,0}(\Sigma, \,\sln(E))$ be a Higgs field; here $\sln(E)$
denotes the subbundle of co-rank one of the endomorphism bundle $\text{End}(E)$ of $E$
defined by the sheaf of endomorphisms of trace zero. Note that $\sln(E)$
is identified with the vector bundle
$\underline{\sln(2, {\mathbb C})}$.

We recall that Hitchin's self-duality equations are 
\begin{eqnarray}\label{SDeq}
F^\nabla +[\Phi\wedge\Phi^*] &=&0 \\
\overline{\partial}^\nabla\Phi &=&0,\nonumber
\end{eqnarray}
where $F^\nabla$ is the curvature of $\nabla$. Here
we write $\overline{\partial}^\nabla$ for the $(0,1)$-part of the connection $\nabla$; note
that $\overline{\partial}^\nabla$
endows $E$ with the structure of a holomorphic vector bundle of rank two.
The second equation means that $$\Phi\,\in\, H^0(\Sigma,\,
K\otimes\sln(E))\, ,$$ where $K$ as before is the holomorphic cotangent
line bundle of the Riemann surface $\Sigma$; the vector bundle $K\otimes \sln(E)$ is
endowed with the holomorphic structure induced by the holomorphic structure on $E$
defined by the above operator $\overline{\partial}^\nabla$ and the natural holomorphic
structure on $K$. 

Let $$\mathcal G\,=\, \Gamma(\SU(E))$$ be the unitary gauge group for $E$ consisting of
$C^\infty$ sections $g\,\in\,
\Gamma(\mathrm{End}(E))$ satisfying the
pointwise identities $g^*g \,=\, \text{Id},\, \det g \,=\,1$.
Then $\mathcal G$ acts on the space of
pairs, consisting of a connection $\nabla$ on $E$ and an element $\Phi$ of $\Omega^{1,0}(\Sigma, \,\sln(E))$, on
the right as follows:
\[
(\nabla,\,\Phi).g \,=\, (g^{-1}\circ \nabla \circ g,\, g^{-1}\Phi g)\, .
\]
The space of solutions of the self-duality equations in \eqref{SDeq} is preserved under
the above action of $\mathcal G$.
A pair $(\nabla,\,\Phi)$ of the above type is called {\it irreducible}
if the following condition holds:

$(\nabla,\,\Phi).g\,=\,(\nabla,\,\Phi)$
only if $g\,\equiv\, \pm \text{Id}_E$.

In other words, the set of irreducible pairs is the locus where the
reduced gauge group $\mathcal G/\{\pm \text{Id}_E\}$ acts freely (note that the action of $\mathcal G$ is
not effective, since $-\text{Id}_E$ acts trivially). The moduli space of solutions
of the self-duality equations is 
$${\mathcal M}_{SD} \,=\,\{(\nabla,\,\Phi) \,\mid\, \text{~solution to \eqref{SDeq}}\}/\mathcal G\, .$$
This is a singular space of real dimension $12(g-1)$ and its smooth locus is
$${\mathcal M}^{irr}_{SD} \,=\, \{(\nabla,\,\Phi) \,\mid\,
\text{~irreducible solution to \eqref{SDeq}}\}/\mathcal G$$
meaning the smooth locus consists of the gauge-equivalence classes of irreducible solutions.

There is a hyper-K\"ahler metric on the manifold ${\mathcal M}^{irr}_{SD}\subset{\mathcal M}_{SD}$ which we will now recall.
Denote by $\mathcal A$ the space of all $\SU(2)$-connections on $E$, and consider the configuration space 
\begin{equation}\label{ca}
{\mathcal C} \,=\, {\mathcal A}\times \Omega^{1,0}(\Sigma,\,\sln(E))\,\cong\, T^*\mathcal A\, .
\end{equation}
Note that we may identify $\mathcal A$ with the space of holomorphic structures on
the vector bundle $\underline{\C^2}$, since
any unitary connection is uniquely determined by its $(0,1)$-part, and
this $(0,1)$-part in turn defines a holomorphic
structure on $\underline{\C^2}$. Conversely, given any holomorphic structure
on $\underline{\C^2}$, there is a unique unitary connection on $\underline{\C^2}$ whose
$(0,1)$-part coincides with the Dolbeault operator on $\underline{\C^2}$ defining the
holomorphic structure; this unique connection is known as the Chern
connection. This $\mathcal C$ is an infinite-dimensional complex manifold, in fact
it is an affine space modelled on the vector space
$\Omega^{0,1}(\Sigma,\,\sln(E))\times \Omega^{1,0}(\Sigma,\,\sln(E))$.

We have a natural K\"ahler metric on $\mathcal C$, compatible with the
above mentioned complex structure, given by the $L^2$-inner
product; indeed, for $(\alpha,\,\phi)\,\in\, \Omega^{0,1}(\Sigma,\,\sln(E))\times
\Omega^{1,0}(\Sigma,\,\sln(E))$, set 
\[
\|(\alpha,\,\phi)\|^2 \,=\, 2\sqrt{-1}\int_\Sigma \mathrm{Tr}(\alpha^*\wedge\alpha +\phi\wedge\phi^*)
\, .
\]
The above mentioned complex structure on $\mathcal C$ is denoted by $I$.
In addition, we have a further compatible complex structure, denoted by $J$, on $\mathcal C$ defined by
\begin{equation}\label{DefJ}
J(\alpha,\,\phi) \,=\, (\sqrt{-1}\phi^*,\,-\sqrt{-1}\alpha^*)\, .
\end{equation}
This complex structure anti-commutes with $I$ and thus setting $K \,=\, IJ$ we have a natural flat
hyper-K\"ahler structure $(\omega_I,\,\omega_J,\,\omega_K)$ on $\mathcal C$. It is
straightforward to check that the complex symplectic form $$\omega_I^\C \,=\,
\omega_J+\sqrt{-1}\omega_K$$ (which is holomorphic with respect to $I$) is just the natural
Liouville symplectic form on the cotangent bundle $\mathcal C \,=\, T^*\mathcal A$: 
\[ \omega_I^\C((\alpha,\,\phi),\,(\beta,\,\psi))\,=\, \int_\Sigma\mathrm{Tr}
(\psi\wedge\alpha - \phi\wedge\beta)\, .
\]

The action of $\mathcal G$ on $\mathcal C$ is compatible with the
above hyper-K\"ahler structure and it turns out that
the self-duality equations in \eqref{SDeq} are precisely the vanishing conditions for the associated
hyper-K\"ahler moment map, once we interpret $\Omega^2(\Sigma,\, \su(E))$ as (a subspace of) the dual of $\mathrm{Lie}(\mathcal G) = \Gamma(\Sigma,\, \su(E))$: 
\begin{eqnarray*}
\mu_I(\nabla,\,\Phi) &=& F^\nabla + [\Phi\wedge\Phi^*]\\
\mu_I^\C(\nabla,\,\Phi) &= &(\mu_J+\sqrt{-1}\mu_K)(\nabla,\,\Phi) = \overline{\partial}^\nabla\Phi.
\end{eqnarray*}
Thus, the moduli space $\mu^{-1}(0)/\mathcal G$ is formally a 
hyper-K\"ahler quotient; in fact, it can be shown, working with appropriate Sobolev completions, 
that its smooth locus ${\mathcal M}^{irr}_{SD} \,\subset\, {\mathcal M}_{SD}$ indeed inherits a hyper-K\"ahler structure. The complex structure on ${\mathcal M}^{irr}_{SD}$ 
corresponding to the complex structure $I$ (respectively, $J$) on $\mathcal C$ will also be 
denoted by $I$ (respectively, $J$); it should not create confusion as the space
under consideration should be clear from the context.

Define the group
\begin{equation}\label{cgg}
\mathcal G^\C \,=\, \Gamma(\Sigma,\,\SL(E))
\end{equation}
of all complex gauge 
transformations of $E$; this $\mathcal G^\C$ consist of all smooth sections $g\,\in\, \Gamma(\mathrm{End}(E))$ 
such that $\det g\,\equiv\, 1$. Then we expect from the standard equivalence between 
complex quotients and symplectic quotients in K\"ahler geometry that the complex 
analytic space $({\mathcal M}_{SD},\,I)$ is isomorphic to the quotient by the 
complex gauge transformations of a certain subset parametrizing polystable points in 
$(\mu_I^\C)^{-1}(0)$. Indeed, Hitchin has shown that the above mentioned equivalence 
between K\"ahler quotients and complex quotients holds in this infinite-dimensional 
setting \cite{Hi1}. The corresponding stability condition is the following.

\begin{definition}
A {\it Higgs bundle} is a holomorphic structure $\overline{\partial}$ on $E$ together with
a holomorphic
section $\Phi\,\in\, H^0(\Sigma,\, K\otimes \sln(E))$. A Higgs bundle $(E,\,\overline{\partial},\,
\Phi)$ is called {\it stable} if every $\Phi$-invariant holomorphic line subbundle $F\,
\subsetneq\, E$ satisfies
the inequality $\deg F\, <\,0$. A Higgs bundle $(E,\,\overline{\partial},\,\Phi)$ is
called {\it polystable} if it is isomorphic to a direct sum of Higgs line bundles of degree
zero.
\end{definition}

Hitchin proved that a Higgs bundle $(E,\,\overline{\partial},\,\Phi)$ is complex gauge 
equivalent to a solution to the self-duality equations if and only if it is polystable \cite{Hi1}. 
Moreover, if $(E,\,\overline{\partial},\,\Phi)$ is a polystable Higgs bundle, then the complex gauge
orbit of 
$(\overline{\partial},\,\Phi)$ contains a single $\mathcal G$-orbit of solutions of the 
self-duality equations. If $(E,\,\overline{\partial},\,\Phi)$ is actually a stable
Higgs bundle, then the corresponding
solutions of the self-duality equations are irreducible.
 In other words, the complex space $({\mathcal M}_{SD},\, I)$
is holomorphically isomorphic to the moduli space of polystable Higgs bundles, and its smooth locus $({\mathcal M}^{irr}_{SD},\, I)$ is biholomorphic to the moduli space of stable Higgs bundles.

To relate the complex manifold $({\mathcal M}_{SD},\,J)$ with the above moduli 
space, observe that the vanishing condition, at $(\nabla,\, \Phi)$,
for the complex moment map with respect 
to $J$ says that the connection $\nabla +\Phi + \Phi^*$ is actually flat. The analogous theorem 
identifying complex and K\"ahler quotients with respect to $J$ was proved by 
Donaldson and Corlette; their theorem says that $({\mathcal M}_{SD},\,J)$ is 
holomorphically isomorphic to the moduli space ${\mathcal M}_{DR}$ of reductive flat 
$\SL(2,\C)$-connections on $\Sigma$, which in turn is biholomorphic to the moduli space of 
reductive representations of the fundamental group $\pi_1(\Sigma)$ into $\SL(2,\C)$. 
Thus, as a complex manifold, ${\mathcal M}_{DR}$ has the following description:
\begin{equation}\label{hims}
{\mathcal M}_{DR} \,= \,({\mathcal 
M}_{SD},\,J)\,=\, \mathrm{Hom}^{red}(\pi_1(\Sigma),\,\SL(2,\C))/\SL(2,\C) \,=:\, \mathcal M_B\, .
\end{equation}
The smooth locus is given by 
the moduli space of irreducible flat $\SL(2,\C)$-connections $ {\mathcal M}^{irr}_{DR}$ (meaning
no line subbundle is preserved by the connection).
The moduli space $\mathcal M_B$ is an affine variety and hence cannot contain any compact complex 
submanifold of positive dimension. In contrast, we have the moduli space of flat
$\SU(2)$-connections as a compact 
subvariety of $({\mathcal M}_{SD},\, I)$ (this subvariety parametrizes
Higgs bundles with zero Higgs field). So the complex structures $I$ and $J$ are 
clearly not equivalent.

\subsection{The twistor space}

Consider the twistor space ${\mathcal M}_{DH}$ of the hyper-K\"ahler 
manifold ${\mathcal M}^{irr}_{SD}$. As a $C^\infty$ manifold we have
\[
{\mathcal M}_{DH} \,= \,{\mathcal M}^{irr}_{SD}\times S^2\,=\, {\mathcal M}^{irr}_{SD}\times 
\{(x,y,z)\,\in\, {\mathbb R}^3\, \mid\, x^2+y^2+z^2\,=\, 1\}\, .
\] 
The quaternionic structure on ${\mathcal M}^{irr}_{SD}$ induces a complex structure $\mathbb I$ on
${\mathcal M}_{DH}$ as follows: Identify $S^2$ with $\C P^1$ by stereographic projection and
choose a standard affine coordinate. Then, at a point $(m,\,\lambda)\,\in\, {\mathcal M}_{DH}$,
the complex structure $\mathbb I$ is given as follows:
\begin{equation}\label{eil}
{\mathbb I}_{m,\lambda} \,= \,I_\lambda\oplus {\mathbf i}\,\in\, \mathrm{End}(T_m
{\mathcal M}^{irr}_{SD}
\oplus T_\lambda\C P^1) \,=\, \mathrm{End}(T_{(m,\lambda)}{\mathcal M}_{DH})\, ,
\end{equation}
where 
\[
I_\lambda \,=\,\frac{1-|\lambda|^2}{1+|\lambda|^2}I + \frac{\lambda + \overline{\lambda}}{1+
|\lambda|^2}J -\frac{\sqrt{-1}(\lambda - \overline{\lambda)}}{1+|\lambda|^2}K
\] 
and $\mathbf i$ is the standard complex structure on $S^2\,=\,\C P^1$. From this
it follows immediately that the natural projection 
\begin{equation}\label{twistorfib}
\pi\,\colon\, {\mathcal M}_{DH}\,= \,{\mathcal M}^{irr}_{SD}\times \C P^1\,\longrightarrow\, \C P^1\,=\, S^2
\end{equation}
is holomorphic.

The manifold ${\mathcal M}_{SD}$ carries an isometric circle (the group
${\rm U}(1)$) action induced by the natural
circle action on the cotangent bundle $\mathcal C\,=\,T^*\mathcal A$ (see \eqref{ca}): 
\[ 
\exp ({\sqrt{-1}\cdot\theta})(\nabla,\,\Phi) \,=\, (\nabla,\, \exp({\sqrt{-1}
\cdot\theta})\Phi)\, .
\]
It is compatible with the complex structure $I$ but rotates $J$ and $K$.
Moreover, this circle action leaves the smooth locus ${\mathcal M}^{irr}_{SD}$ invariant.
In other words, we obtain
a holomorphic circle action on $({\mathcal M}_{DH},\,{\mathbb I})$ covering the standard
circle action on $\C P^1$ via 
\[
\exp ({\sqrt{-1}\theta})(\nabla,\,\Phi,\,\lambda) \,=\, (\nabla,\, \exp({\sqrt{-1}\theta})
\Phi,\, \exp({\sqrt{-1}\theta})\lambda)\, , \ \ \theta\, \in\, \mathbb R\, .
\]
It can be shown that this action extends to an action of $\C^*$ on ${\mathcal M}_{DH}$ covering the 
standard $\C^*$-action on $\C P^1 \,=\, \C\cup \{\infty\}$
given by multiplication. As a consequence of this action, for any 
$\lambda\,\neq\, 0,\,\infty$, the complex manifold $({\mathcal M}^{irr}_{SD},\,I_\lambda)$
(see \eqref{eil} for $I_\lambda$) is actually biholomorphic to $({\mathcal M}^{irr}_{SD},\,J)$ (note
that $J\, =\,I_1$). In order to describe these biholomorphisms
explicitly it is customary to use the $\lambda$-connections.

\subsection{$\lambda$-connections and the Hodge moduli space}\label{sec:lambdaConn}

Let us examine more closely the twistor space constructed above.
Take any $\lambda\, \in\, \mathbb C$. The complex structure $I_\lambda$ on
${\mathcal M}^{irr}_{SD}$ is the restriction of the complex structure $\widehat{I}_\lambda$
on $\mathcal C$ that sends
a tangent vector $$(\alpha,\,\phi)\,\in\, \Omega^{0,1}(\Sigma,\,\sln(E))\times
\Omega^{1,0}(\Sigma,\,\sln(E))$$ to
\[
\widehat{I}_\lambda(\alpha,\,\phi) \,=\, \frac{1}{1+|\lambda|^2}
\left((1-|\lambda|^2)\sqrt{-1}\alpha + 2\sqrt{-1}\lambda \phi^*,\, 
(1-|\lambda|^2)\sqrt{-1}\phi - 2\sqrt{-1}\lambda \alpha^* \right)\, .
\]
From this formula it follows that the following $\R$-linear map
$$(\Omega^{0,1}(\Sigma,\,\sln(E))\times \Omega^{1,0}(\Sigma,\,\sln(E)),\, \widehat{I}_\lambda)
\, \longrightarrow\,
(\Omega^{0,1}(\Sigma,\,\sln(E))\times\Omega^{1,0}(\Sigma,\,\sln(E)),\,\sqrt{-1})$$
\begin{equation}\label{ca2}
(\alpha,\phi)\,\longmapsto \,(\alpha,\,\phi)^\lambda \,:=\, (\alpha + \lambda\phi^*,\,\phi-\lambda\alpha^*)
\end{equation}
intertwines the complex structures.

\begin{definition}
A {\it generalized $\lambda$-connection} on $E\,=\,\underline{\C^2}$ consists of the following:
\begin{itemize}
\item a holomorphic structure on $E$
such that the corresponding holomorphic structure $\overline\partial$ on the trivial line bundle
$\det E\,=\, \underline{\C}$ coincides with the one given by the trivialization, in other
words, $\overline\partial$ coincides with the $(0,1)$-part of the de Rham differential
(hence the notation), and

\item a $C^\infty$ differential operator
\[
D\,\colon\, \Gamma(E)\,\longrightarrow\, \Omega^{1,0}(E)
\]
satisfying 
\[
D(fs) \,=\, \lambda\partial f\otimes s + fDs
\]
such that the differential operator on $\det E\, =\, \bigwedge^2 E$ it induces actually coincides with $\lambda \partial$.
\end{itemize}
A generalized $\lambda$-connection of the above type is called a {\it $\lambda$-connection}
if the differential operator $D$ is holomorphic (not just $C^\infty$), in other
words, if it satisfies the equation
\begin{equation}\label{dd-dd0}
D\overline{\partial}+\overline{\partial} D\,=\,0\, .
\end{equation}
A $\lambda$-connection is called {\it irreducible} if it does not admit any $D$-invariant holomorphic 
line subbundle.
\end{definition}

Let $\overline{\mathcal A}^\lambda$
denote the space of generalized $\lambda$-connections on $E$.
Note that if $\lambda\,=\,0$ then $\overline{\mathcal A}^\lambda$ is
the space of all pairs consisting of a holomorphic structure on $E$ together with
a $C^\infty$ $(1,0)$-form with values in the vector bundle $\sln(E)$.

The complex linearity of the map $(\alpha,\,\phi)\,\longmapsto\,
(\alpha,\,\phi)^\lambda$ in \eqref{ca2} implies that for any fixed
$\lambda \,\in\,\C$, there is
a holomorphic identification of complex manifolds
\[
(\mathcal C,\,I_\lambda)\,\longrightarrow\, \overline{\mathcal A}^\lambda\, ,
\quad (\nabla,\,\Phi)\,\longmapsto\,
(\overline{\partial}^\nabla + \lambda\Phi^*,\, \lambda\partial^\nabla + \Phi)\, ,
\]
where $\mathcal C$ is defined in \eqref{ca};
recall that an $\SU(2)$-connection on $E$ is uniquely determined by its component
of type $(0,1)$.

As $\mathcal M_{DH}$ is the twistor space of the hyper-K\"ahler manifold $\mathcal M_{SD}^{irr}$ it comes with a holomorphic section 
\[\omega^\C\,\in\, H^0(\mathcal M_{DH},\,\Lambda^2 T^*_F\otimes \pi^*{\mathcal O}_{\C P^1}(2)),\]
where $T_F\,=\, {\rm kernel}(d\pi) \,\subset\, T\mathcal M_{DH}$ is the tangent bundle along the fibers of $\pi$ (see also \cite{HKLR} for the general setup).

For every fixed $\lambda\,\in\, \mathbb C$, this section gives the associated holomorphic symplectic form
$\omega_\lambda^\C$ on $\pi^{-1}(\lambda) \,=\, ({\mathcal M}^{irr}_{SD},\, I_\lambda)$ given by 
\[
\omega_\lambda^\C \,=\, \omega_I^\C + \lambda\omega_I +
\lambda^2\overline{\omega_I^\C}\, ,
\]
while the complex moment map $\mu_\lambda^\C$ is given by 
\[
\mu_\lambda^\C \,= \,\mu_I^\C + \lambda \mu_I + \lambda^2\overline{\mu_I^\C}\, .
\]
It follows that the vanishing condition for the complex moment map $\mu_\lambda$ with
$\lambda\,\neq\, 0$ is equivalent to the vanishing of the curvature of the connection 
\[
\nabla^\lambda \,=\, \nabla + \lambda\Phi^* + \lambda^{-1}\Phi\, .
\]
This flatness condition can be interpreted as the integrability condition (as in \eqref{dd-dd0})
of
\[
\mathcal D(\lambda) \,=\, \lambda\partial^\nabla + \Phi
\]
on $E$ with respect to the holomorphic structure 
\[\overline{\partial}^\lambda \,=\, \overline{\partial}^\nabla + \lambda\Phi^*\, .\] 
From the general description
of the twistor space of a hyper-K\"ahler quotient (see \cite{HKLR}) we expect that, at least
formally, the fiber $\pi^{-1}(\lambda)$ (see \eqref{twistorfib} for the map $\pi$) is the
symplectic quotient of $(\mathcal C,\,I_\lambda, \,
\omega_\lambda^\C)$ by the complexified gauge group $\mathcal G^\C$
(defined in \eqref{cgg}), consequently, $\pi^{-1}(\lambda)$ is the
moduli space of (holomorphic) $\lambda$-connections. Note that $\pi^{-1}(0)$
is the moduli space of stable Higgs bundles, because a holomorphic $0$-connection is just a
Higgs field.

Let ${\mathcal M}_{Hod}$ 
be the moduli space of triples of the form $(\overline{\partial},\,\mathcal D,\,\lambda)$ consisting of a 
holomorphic structure $\overline{\partial}$ on $E$, a complex number $\lambda$ and an irreducible 
(holomorphic) $\lambda$-connection $\mathcal D$.
In view of the above discussion, we have a natural isomorphism \[ \pi^{-1}(\C)\,\cong\,{\mathcal 
M}_{Hod},\quad (\nabla,\,\Phi,\,\lambda)\,\longmapsto\, (\overline{\partial}^\lambda,\,\mathcal 
D(\lambda),\,\lambda)\, , \] where $\pi$ is given in \eqref{twistorfib}.

The $\C^*$-action on $\mathcal M_{DH}$ translates under this identification into the action
$$t\cdot (\overline{\partial},\, \mathcal D,\,\lambda)\,=\, (\overline{\partial},\,
t\mathcal D,\, t\lambda)\, , \ \ t\,\in\,\C^*\, .$$ In particular, there are natural biholomorphisms
\[
\pi^{-1}(\C^*) \,\cong\, \C^*\times \mathcal M^{irr}_{DR} \,\cong\, \C^*\times \mathcal M^{irr}_{B}
\]
(see \eqref{hims}).

\subsection{Involutions and real structures}\label{sec:Involutions}

The twistor space $\mathcal M_{DH}$ is equipped with a number of structures,
some of which will be recalled here. All maps below involving $\mathcal M_{DH}$ will
be defined with respect to the smooth trivialization $\mathcal M_{DH}\,=\,
\mathcal M^{irr}_{SD}\times \C P^1$.

First, we have the holomorphic involution 
\begin{equation}\label{N}
N\,:\, \mathcal M_{DH}\,\longrightarrow\,\mathcal M_{DH}\, ,\quad (\nabla,\,\Phi,\,\lambda)
\,\longmapsto\, (\nabla,\,-\Phi,\,-\lambda)\, .
\end{equation}
This is just $-1\,\in\, {\rm U}(1)\,=\, S^1$ acting
in the context of the $S^1$-action described above. Note that the involution of
$\mathcal M^{irr}_{SD}$ given by $(\nabla,\,\Phi)\,\longmapsto\, (\nabla,\,-\Phi)$ is holomorphic with
respect to the complex structure $I$ while being anti-holomorphic with respect to
the complex structure $J$. The fixed point locus for this involution
corresponds to the space of solutions $(\nabla,\,\Phi)$ of the self-duality equations such that
there exists a gauge transformation $g\,\in\,\mathcal G$ satisfying the following two
equations:
\[
\nabla.g \,=\, \nabla,\quad g^{-1}\Phi g \,=\, -\Phi\, .
\]

Furthermore, we have the following standard real structure $\tau$ on $\mathcal M_{DH}$, which is
there on any twistor space of a hyper-K\"ahler manifold covering the antipodal map:
\begin{equation}\label{et}
\tau\,:\, \mathcal M_{DH}\,\longrightarrow\,\mathcal M_{DH},\quad
(\nabla,\,\Phi,\,\lambda) \,\longmapsto\, (\nabla,\,\Phi,\,-
\overline{\lambda}^{-1})\, .
\end{equation}

The two maps $N$ and $\tau$ constructed in \eqref{N} and \eqref{et} respectively, actually commute;
in fact, they combine to give a third anti-holomorphic involution $\rho$ defined as 
\begin{equation}\label{defC}
\rho \,:= \,\tau\circ N\,:\, \mathcal M_{DH}\,\longrightarrow\,\mathcal M_{DH}\, ,\quad 
(\nabla,\,\Phi,\,\lambda) \,\longmapsto\, (\nabla,\,-\Phi,\,
\overline{\lambda}^{-1}).
\end{equation}

In terms of the $\lambda$-connections, the involution
$N$ constructed in \eqref{N} becomes the involution
$$(\overline{\partial}, \,\mathcal D,\,\lambda)\,\longmapsto\,
(\overline{\partial}, \,-{\mathcal D},\,-\lambda)\, .$$

The anti-holomorphic involution $\tau$ in \eqref{et}
acts on $\pi^{-1}(\C^*)\,\subset\, \mathcal M_{Hod}$ by the rule
\[
\tau(\overline{\partial}^\lambda, \,\mathcal D(\lambda),\,\lambda) \,=\, (\overline{\partial}^\nabla -
\overline{\lambda}^{-1} \Phi^*,\, -\overline{\lambda}^{-1}\partial^\nabla + \Phi,
\,-\overline{\lambda}^{-1}) \,=\, (\overline{\lambda}^{-1}\overline{\mathcal D(\lambda)}^*,\, -\overline{\lambda}^{-1}\overline{(\overline{\partial}^\lambda
)}^* ,\,-\overline{\lambda}^{-1})\, .
\]

The right-hand side should be interpreted as follows.
The standard Hermitian structure on ${\mathbb C}^2$ identifies $(\overline{\mathbb C^2})^*$.
This isomorphism produces an isomorphism $E\,\cong\,\overline{E}^*$, because
$E\,=\,\underline{\C^2}$. Using this isomorphism, $\nabla^\lambda$ transforms into
the flat connection
$$\overline{\nabla^\lambda}^* \,=\,
\partial^\nabla -\overline{\lambda} \Phi + \overline{\partial}^\nabla 
- \overline{\lambda}^{-1}\Phi^*\, ,$$ which is a 
$(-\overline\lambda^{-1})$-connection on the Riemann surface $\Sigma$. 
In terms of the trivialization $$\pi^{-1}(\C^*)\,\cong\, \mathcal M_{B} \times\C^*$$
we may express $\tau$ as
$$\tau(R,\,\lambda) \,=\, (\tau'(R),\,-\overline{\lambda}^{-1})\, .$$ It follows
that $\tau'$ is the anti-holomorphic involution of $\mathcal M_B$ that
sends a representation $R$ to its conjugate transpose, in other words,
\[
\tau(\rho)(\gamma)\,=\, \left(\overline{R(\gamma)}^{-1}\right)^t\, ;
\]
this is in accordance with Simpson's discussion in \S~4 of \cite{Si}.

\subsection{The moduli space}

It can be shown that $(\mathcal M^{irr}_{SD}(\Sigma),\,-I)$ is isomorphic to
$(\mathcal M^{irr}_{SD}(\overline{\Sigma}),\,I)$, or in other words, $$\pi^{-1}(\infty) 
\,=\, \mathcal M_{Higgs}(\overline{\Sigma})\, .$$ Thus, applying the same construction as above
with $\Sigma$ replaced by $\overline\Sigma$, we get that
\[
\pi^{-1}(\C P^1\setminus\{0\}) \,\cong\, \mathcal M_{Hod}(\overline{\Sigma})\, ,
\quad (\nabla,\,\Phi,\,\lambda)\,\longmapsto\, (\partial^\nabla + \lambda^{-1}\Phi,\,
\lambda^{-1}\overline{\partial}^\nabla + \Phi^*,\,\lambda^{-1})\, ,
\]
where $\pi$ is the projection in \eqref{twistorfib}
(note that $\Omega^{1,0}(\Sigma) \,=\, \Omega^{0,1}(\overline{\Sigma})$); the
above term $$(\partial^\nabla + \lambda^{-1}\Phi,\,
\lambda^{-1}\overline{\partial}^\nabla + \Phi^*,\,\lambda^{-1})$$
is a holomorphic vector bundle on $\overline\Sigma$ equipped
with a $\lambda^{-1}$-connection. 
We can thus think of the twistor space $\mathcal M_{DH}$ as the union of the two Hodge-moduli spaces
$\mathcal M_{Hod}(\Sigma),\,\mathcal M_{Hod}(\overline{\Sigma})$ glued together
along the inverse images of $\C^* \,=\,
\C P^1\setminus\{0,\,\infty\}$ via 
\[
(\overline{\partial},\,{\mathcal D},\,\lambda) \,\sim\, (\lambda^{-1}\mathcal D, \,
\lambda^{-1}\overline{\partial}, \,\lambda^{-1})
\]
for any $(\overline{\partial},\,{\mathcal D},\,\lambda)\,\in\,
\mathcal M_{Hod}(\Sigma)$ with $\lambda\,\neq\, 0$.
The anti-holomorphic involution $\tau$ in
\eqref{et} acts on $$\pi^{-1}(\C^*)\,\subset\, \mathcal M_{Hod}$$ by 
\[
\tau(\overline{\partial}^\lambda,\, \mathcal D(\lambda),\,\lambda)
\,=\, (\overline{\partial}^\nabla -\overline{\lambda}^{-1} \Phi^*,\, -
\overline{\lambda}^{-1}\partial^\nabla + \Phi,\, -\overline{\lambda}^{-1})\,\sim\,
(\partial^\nabla -\overline{\lambda} \Phi,\,
-\overline{\lambda}\overline{\partial}^\nabla + \Phi^*,\, -\overline{\lambda})\, .
\]
Under the earlier mentioned isomorphism $E\,\cong\,\overline{E}^*$ given by the standard Hermitian
structure on ${\mathbb C}^2$, we may interpret $$(\partial^\nabla -\overline{\lambda} \Phi,\,
-\overline{\lambda}\overline{\partial}^\nabla + \Phi^*,\, -
\overline{\lambda})$$ as a holomorphic structure with
$(-\overline{\lambda})$-connection on
$\overline{E}^*$, viewed as a holomorphic bundle on $\overline\Sigma$.

\subsection{Special features of the rank $2$ case}\label{sec:Conj_vs_adj}

The above discussions are valid for higher ranks. In the rank two case under
consideration, the standard complex symplectic form on ${\mathbb C}^2$
produces a complex symplectic form on the fibers of the vector bundle
$E \,=\, \underline{\C^2}$; this complex symplectic form on $E$ will be denoted
by $\Omega$. We note that $\Omega$ produces an isomorphism $E\,\cong\, E^*$.
To describe it explicitly, we may take 
\[
\Omega \,=\, \epsilon_1\wedge \epsilon_2\, ,
\] 
where $\epsilon_1,\,\epsilon_2$ is the dual basis to the standard orthonormal basis
$e_1,\,e_2$ of $\C^2$. Thus, $\overline{E}^*$ is isomorphic to $\overline E$. More generally,
any Higgs bundle, with $E$ as the underlying bundle, is isomorphic to the
dual Higgs bundle. Let $(\nabla,\,\Phi)$ be a solution to the self-duality equations. Invoking the
identification $$(\overline{\partial}^\nabla,\,
\mathcal D(\lambda),\,\lambda) \,\longleftrightarrow\, \nabla^\lambda$$ we shall often
interpret $\tau(\nabla,\,\Phi,\,\lambda) = (\nabla,\,\Phi,\,-\overline{\lambda}^{-1})$ as $\overline{\nabla^{-\overline{\lambda}^{-1}}}$, where
$\tau$ is constructed in \eqref{et}. 

The above construction can be explained as follows. Consider
\[\delta\,=\,\begin{pmatrix} 0 & 1 \\ -1 & 0 \end{pmatrix} ,\]
the Gram-matrix of $\Omega$. 
Then, for any $A\,\in\,\sln(2,\C)$ we have by a short calculation that 
\[-A^*\,=\, -\overline{A}^t\,=\, \delta^{-1}\overline{A}\delta\,.\]
Therefore, viewing $\delta$ as a unitary gauge transformation, a unitary connection $\nabla$ satisfies
the equation 
\[\nabla\,\cong\, \nabla^{\overline{E}^*}\,=\,\overline{\nabla}.\delta\, ,\]
and the family of flat connections corresponding to a solution $(\nabla,\,\Phi)$ of the self-duality equations can be written as
\[\nabla^\lambda\,=\,\nabla+\lambda^{-1}\Phi+\lambda \Phi^*
\,=\,\nabla+\lambda^{-1}\Phi-\lambda\delta^{-1}\overline{\Phi}\delta\, .\]
Taking complex conjugates and gauging with $\delta$ using $\delta^2 \,=\,
-\mathrm{id}$ yields the following:
\begin{equation}\label{SD_real}
\begin{split}
\overline{\nabla^\lambda}.\delta &= \overline{\nabla}.\delta+\overline{\lambda}^{-1}\delta^{-1}\overline{\Phi}\delta - \overline{\lambda} \delta^{-2} \Phi\delta^2\\
&= \nabla - \overline{\lambda}^{-1}\Phi^*-\overline{\lambda}\Phi \\
&= \nabla^{-\overline{\lambda}^{-1}}.
\end{split}
\end{equation}

\section{Basic properties of sections of $\mathcal M_{DH}$}\label{sec2}

\begin{definition}
A holomorphic section $$s\, :\, {\mathbb C}P^1\,\longrightarrow\,
{\mathcal M}_{DH}$$ of the fibration $\pi\, :\, {\mathcal M}_{DH}\,\longrightarrow\,
{\mathbb C}P^1$ is called {\it irreducible} if
\begin{itemize}
\item the $\lambda$-connection $s(\lambda)$ is irreducible for every
$\lambda\,\in\,\mathbb CP^1\setminus\{0,\, \infty\}$, and

\item the Higgs bundles $s(0)$ and $s(\infty)$ over $\Sigma$ and $\overline\Sigma$
respectively are stable.
\end{itemize}
\end{definition}

Note that the image of an irreducible section $s$ is contained in the smooth locus 
of $\mathcal M_{DH}$; see \cite{Si}. Instead of working with smooth 
$\lambda$-connections, it is sometimes necessary to deal with (integrable) 
$\lambda$-connections which are only of class $C^k$, and there is no
essential difficulty in doing so. In fact, we
are only dealing with $\lambda$-connections of class $C^k$ which 
are gauge equivalent (by gauge transformations of class $C^{k+1}$) to some smooth 
$\lambda$-connections.

\begin{lemma}\label{lift-section}
Let $s$ be an irreducible holomorphic section of the fibration $$\pi\,:\,\mathcal M_{DH}\,
\longrightarrow\, \mathbb CP^1\, .$$
For every $k\,\in\,\mathbb N^{\geq2}$, there exists a holomorphic lift $\widehat s$ of $s$ on
$\mathbb C\,\subset\,\mathbb CP^1$ to the space of $\lambda$-connections of class $C^k$.
\end{lemma}

\begin{proof}
Locally on sufficiently small open subsets of $\C$ there are holomorphic lifts of $s$ to the 
infinite-dimensional space of $\lambda$-connections, by the construction of the 
Hodge moduli space. These local lifts give then a cocycle with values in the Banach 
Lie group of gauge transformations of class $C^k$ (where we can choose $k$ 
to be arbitrarily large). Note that after this choice of local lifts, the cocycle is 
uniquely determined as the section $s$ is irreducible. Since $\mathbb C$ is a Stein 
manifold, the generalized Grauert theorem, \cite{Bu}, tells us that the cocycle is 
trivial, which enables us to find a global lift on $\mathbb C$; see the proof of 
Theorem 8 in \cite{He3} for more details.
\end{proof}

\begin{remark}\label{conne-lift}
Instead of working with the lift ${\widehat s}(\lambda)\,=\,(
\overline{\partial}(\lambda),
\,\mathcal D(\lambda),\, \lambda)$ we will often work with the corresponding family of flat connections
\[\lambda\,\longmapsto\,\nabla^\lambda\, :=\, \overline{\partial}(\lambda)+\lambda^{-1}\mathcal D(\lambda)\, .\]
\end{remark}

\subsection{Admissible Sections}

\begin{definition}\label{admissible}
We call a holomorphic section $s$ of $\mathcal M_{DH}$ {\it admissible} if it admits a
lift $\widehat s$ on $\mathbb C$ of the form
\[{\widehat s}(\lambda)\,=\,(\overline{\partial}+\lambda \Psi,\,\lambda\partial+\Phi,\, \lambda)\]
for a Dolbeault operator $\overline\partial$ of type $(0,1)$, a
Dolbeault operator $\partial$ of type $(1,0)$, a $(1,0)$-form $\Phi$
and a $(0,1)$-form $\Psi$, such that $(\partial,\Psi)$ is a semi-stable Higgs pair on $\overline\Sigma.$
\end{definition}

\begin{remark}
By definition of a lift, $\widehat s(0)\,=\,(\overline{\partial},\,\Phi)$ is semi-stable as it
represents a point in the moduli space
of semi-stable Higgs bundles on $\Sigma$ -- the fiber over $\lambda\,=\,0.$ The lift $\widehat s$ on
$\mathbb C$
yields a lift of $s$ on $\mathbb CP^1\setminus\{0\}$ via
\[\widetilde{\lambda}\,\longmapsto\, (\partial+\widetilde{\lambda}\Phi,
\,\widetilde{\lambda}\overline{\partial}+ \Psi, \, \widetilde{\lambda})_{\overline\Sigma}\, ,\]
where $\widetilde{\lambda}\,=\,1/\lambda.$ Hence, the above definition actually treats $0$ and $\infty$ in a symmetric way.
\end{remark}

\begin{example}\label{ex:basic}\mbox{}
\begin{enumerate}[(i)]
\item A solution to Hitchin's self-duality equations induces an admissible section $s$; see the discussion
in Section \ref{se1}. 
Moreover, this section is irreducible if it is irreducible at some (any) point $\lambda\,\in\,
\mathbb CP^1$; see \cite{Hi1}.

\item A harmonic map $f\,\colon \,\Sigma\,\longrightarrow\, {\rm SU}(2)$ gives rise to a family of flat
connections of the form $\nabla + \lambda^{-1} \Phi - \lambda\Phi^*$, where $\nabla$ is the pullback of the Levi-Civita
connection on $S^3\,\cong\, {\rm SU}(2)$ and $\Phi-\Phi^* \,=\, 2f^{-1}df$ (see, e.g. \cite{Hi}). The flatness of the family $\nabla^\lambda$ is equivalent to $(\nabla,\,\Phi)$ solving the signed self-duality equations
\begin{equation*}
\begin{split}
F^\nabla- [\Phi\wedge\Phi^*] \, &=\, 0\\
\overline{\partial}^\nabla\Phi\, &=\, 0\, .
\end{split}
\end{equation*}
If $f$ is a minimal surface which is not contained in a totally geodesic $S^2\,\subset\, {\rm SU}(2)$, and the genus of
$\Sigma$ is at least two, then this family of flat connections
is generically stable and the Higgs pair $(\overline{\partial}^\nabla,\,\Phi)$ is stable \cite{He}. Thus, such a minimal surface induces an admissible section
as well.
However, such a section is not irreducible, because it gives the trivial gauge class at $\lambda\,=\, \pm1.$
\end{enumerate}
\end{example}

We are now going to investigate under which conditions an irreducible section $s$ is admissible. To that end we 
consider lifts of $s$ on $\mathbb C\,\subset\,\mathbb CP^1$ and on $\mathbb CP^1\setminus\{0\}\,\subset\,\mathbb CP^1$ given by families of flat connections of the form
\[^+\nabla^\lambda\,=\,\lambda^{-1}\Phi_{-1}+\nabla+\lambda\Phi_1+\cdots\]
and
\[^-\nabla^\lambda\,=\,\lambda\Psi_1+\widetilde\nabla+\lambda^{-1}\Psi_{-1}+\cdots\, . \]
These are provided by Lemma \ref{lift-section} and Remark \ref{conne-lift}. By
construction and irreducibility, there exists, locally, a holomorphic $\lambda$-family of
$\text{GL}(2,\mathbb C)$-gauge transformations $g$ defined on simply connected subsets $\lambda\,\in\, U\,\subset\,\mathbb C\setminus\{0\},$ gauging $^+\nabla^\lambda$ to $^-\nabla^\lambda$.
This family of gauge transformations
is unique up to scaling by a holomorphic function only depending on $\lambda$; indeed,
this is a consequence of irreducibility.
This determines a holomorphic line bundle over $\mathbb C\setminus\{0\}$ whose
fibers are holomorphic (in fact parallel) endomorphisms with respect to the holomorphic structures
induced by $^\pm\nabla^\lambda$. Since any line bundle over $\C \setminus\{0\}$ is
trivial, we may consider a holomorphic section $g$ thereof without zeros. Hence
$g(\lambda)$ is a gauge transformation such that we have
\begin{equation}\label{ge2}
^+\nabla^\lambda.g(\lambda) \,=\,\, ^-\nabla^\lambda
\end{equation}
for all $\lambda\,\in\,\mathbb C\setminus \{0\}$, and the map defined by
$\lambda\,\longmapsto\, g(\lambda)$ is holomorphic.

For any point $x\,\in\,\Sigma$ we get an element $g_x$ in the loop group
$L\mathrm{GL}(2,\mathbb C)$ by setting $$g_x(\lambda)\,=\, g(x,\,\lambda)$$ (see 
\cite{PrSe} for details on the theory of loop groups). Let us assume for a moment that
for any $x\,\in\,\Sigma$ the loop $g_x$ is in the big cell, meaning it has a factorization
\[g_x\,=\,g_x^+ g_x^-\]
such that $g_x^+$ and $g_x^-$ extend holomorphically to $\lambda\,=\,0$ and $\lambda\,=\,\infty$
respectively, in $\text{GL}(2,\mathbb C)$. After a normalization (see Theorem 8.1.2 in
\cite{PrSe}), we get a global (on $\Sigma)$ smooth factorization 
\[g\,=\,g^+ g^-\]
into positive and negative $\lambda$-dependent gauge transformations.

Now, the equation in \eqref{ge2} is equivalent to the equation
\[^+\nabla^\lambda.g^+(\lambda)\,=\, ^-\nabla^\lambda.(g^-)^{-1}(\lambda)\, .\]

Expanding both sides of the above equation in powers of $\lambda$, we see that the left hand side of the
above equation has only one $\lambda^{-1}$-term, while the right hand side has
only one $\lambda$-term. Hence, under the above assumption of $g_x$ lying in the big
cell for all $x\,\in\,\Sigma$, we have an admissible section.

\begin{proposition}\label{strange-prop}
Let $^+\nabla^\lambda$ and $^-\nabla^\lambda$ be lifts of an irreducible section $s$
of $\mathcal M_{DH}$ over $\mathbb C$ and $\mathbb CP^1\setminus\{0\}
\,\subset\,\mathbb CP^1$ respectively, and let $g$ be the ${\rm GL}(2,\mathbb C)$-valued $\lambda$-dependent gauge transformation between them.
Then, the property of $g_x$ being in the big cell for all $x\,\in\,\Sigma$ does not
depend on the choice of lifts $^+\nabla^\lambda$ and $^-\nabla^\lambda$. 
If this property is satisfied, then $s$ is admissible.
\end{proposition}

\begin{proof}
The proposition follows from just unraveling the definitions: Choose different lifts
$$^+\widehat{\nabla}^\lambda \,= \,^+\nabla^\lambda.g_1(\lambda)\ \ \text{ and }\ \
^-\widehat{\nabla}^\lambda \,=\, ^-\nabla^\lambda.g_2(\lambda)\, ,$$
with $g_1$ defined for all $\lambda\,\in\, \C$ and $g_2$ defined for all $\lambda
\,\in\, \C P^1\setminus\{0\}$. It follows that $\widehat{g} \,=\, g_1^{-1} g g_2 $ is
a gauge transformation satisfying the condition that $^+\widehat{\nabla}^\lambda.\widehat{g}(\lambda)
\,= \, ^-\widehat{\nabla}^\lambda$ and also lying in the big cell for all $x
\,\in\, \Sigma$. Indeed, we have $\widehat{g} \,=\, \widehat{g}^+\widehat{g}^-$, where
$\widehat{g}^+ \,=\, g_1^{-1} g^+$ and $\widehat{g}^- \,=\, g^-g_2$.
\end{proof}

Note that in the case of admissible sections, we can just take $g(\lambda) \,= \,
\mathrm{id}_E$ for the holomorphic family of gauge transformations. In particular, 
this gauge transformation is $\text{SL}(2,{\mathbb C})$-valued.

For a general irreducible section $s$ of $\mathcal M_{DH}$, we obtain a ${\mathbb 
Z}/ 2{\mathbb Z}$--valued invariant of $s$, which will be denoted by 
$\text{par}(s)$, as follows: Consider lifts $^+\nabla^\lambda$ and $^-\nabla^\lambda$ 
as above and choose a $\text{GL}(2,\mathbb C)$-valued $\lambda$-dependent gauge 
transformation $g$ such that we have $$^+\nabla^\lambda.g(\lambda) \,= \,^-\nabla^\lambda$$
for all $\lambda\,\in\, \C\setminus\{0\}$. The determinant map $\lambda\,\longmapsto
\,\det g(\lambda)$ is independent of $p\,\in\, \Sigma$ by irreducibility, and hence
it defines a holomorphic function from $\mathbb C\setminus\{0\}$ to itself.

If we replace $g$ by $fg$, where $f\,:\,\C\setminus\{0\}\,\longrightarrow\,
\C\setminus\{0\}$ is a 
holomorphic function, then the degree of the determinant map is multiplied by an 
element of $2\mathbb Z$. So the parity of the degree of determinant maps for $f$ and
$fg$ coincide. The parity of the degree of the determinant map is 
moreover independent of the choice of lifts $^\pm\nabla^\lambda$. (It is well-defined by
the above observations.)

\begin{definition} 
Let $s$ be an irreducible section of $\mathcal M_{DH}$ with choice of lifts
$^\pm\nabla^\lambda$. Then the {\it parity} of $s$, denoted by $\mathrm{par}(s)\,\in\,
{\mathbb Z}/2{\mathbb Z}$ is defined to be the parity of the degree of the
determinant map $\lambda\,\longmapsto\, \det g(\lambda)$, where $g$ is an arbitrary 
family of gauge transformations satisfying $^+\nabla^\lambda.g(\lambda) \,
= \,^-\nabla^\lambda$. 
\end{definition}

It is clear from the above discussion that an admissible section must have zero 
parity, as we may choose $g$ to be $\mathrm{SL}(2,\C)$-valued.

\begin{definition}\label{vsw}
A holomorphic vector bundle $(E,\,\overline\partial)$ on $\Sigma$ is called {\it very stable} if it admits no
nonzero nilpotent Higgs field.
A stable vector bundle which is not very stable is called {\it wobbly} \cite{DP}.

A {\it wobbly Higgs bundle} is a Higgs bundle $(E,\,\overline\partial,\, \Phi)$ such that the holomorphic bundle $(E,\,\overline\partial)$ is stable and the Higgs field $\Phi$ nilpotent and nonzero.
\end{definition}

The notion of very stable bundle was introduced
by Laumon in \cite{Lau}. It is known that a very stable vector bundle is stable \cite{Lau}.

We first show the existence of wobbly Higgs bundles. In other words, we show the existence of
stable Higgs pairs $(E,\,\overline\partial,\,\Phi)$, 
where $(E,\,\overline\partial)$ is stable of degree zero, and $\Phi$ is nilpotent and nonzero. This is well known; however, since wobbly
Higgs bundles play an important role in this paper (see Proposition \ref{Prop:OddPar}, Theorem \ref{countertheorem}),
we include a proof to be somewhat self-contained. It may be mentioned that it is much more difficult to
show that there are very stable vector bundles \cite{Lau}, despite very stability being a Zariski
open condition.

\begin{proposition}\label{StableNilpotent}
Let $\Sigma$ be a compact Riemann surface of genus $g\,\geq\, 2$. Then there exists a wobbly
vector bundle $(E,\,\overline\partial)$ on $\Sigma$ of rank $2$ and trivial determinant. (In other words,
$(E,\,\overline\partial)$ is stable and admits a non-trivial Higgs field $\Phi\,\in\, H^0(\Sigma,\, K\otimes \mathfrak{sl}(E))$
which is nilpotent.)
\end{proposition}

\begin{proof}
The hyperelliptic case $g\,=\,2$ follows from the results in \cite{NR} or can be dealt 
with using the methods of \S~5 of \cite{He}.

Assume that $g\,\geq\, 3$ and pick a point $p\,\in \,\Sigma$. Let $L \,=\, L(p)$ be the
line bundle associated with the divisor $p$; so $L$ has a distinguished 
section $s_p\,\in\, H^0(\Sigma,\, L)$ which has a simple zero exactly at $p$. We construct the bundle $V$ as a non-trivial extension of $L$ by $L^{-1}$: As 
a smooth bundle we shall take $E\,=\, L^{-1} \oplus L$ with holomorphic structure $\overline\partial$ of the form
\[
\overline{\partial} \,=\, \begin{pmatrix} \overline{\partial}_{L^{-1}} & \gamma \\ 0 &
\overline{\partial}_L\end{pmatrix},
\]
with $\gamma\in \Gamma(\overline{K}\otimes L^{-2})$ such that $0\,\neq\, \gamma \,\in\, H^1(\Sigma,\,L^{-2})$ to be 
chosen later.

By Riemann-Roch, and using $h^0(L^2)\,>\,0$, we have
\[
h^0(K\otimes L^{-2}) \,=\, 2g-2 - 2 + 1-g + h^0(L^2) \,>\, g-3 \,\geq\, 0\, ,
\]
consequently, there exists a non-trivial $\alpha\,\in\, H^0(\Sigma,\,K\otimes L^{-2})$, and we define 
\[
\Phi \,=\, \begin{pmatrix}0 & \alpha\\ 0 & 0\end{pmatrix},
\]
which is a nilpotent holomorphic section of $K\otimes \mathrm{End}(E)$ for any choice of $\gamma$. We may interpret $\alpha$ as a
holomorphic section of $K$ vanishing to order at least two at $p$. 

Now let $V\subset E$ be a holomorphic rank $1$ subbundle with $\mathrm{degree}(V)
\,\geq\, 0$. With respect to the smooth splitting $E\,=\,L^{-1}\oplus L$ we may write the
inclusion $\iota\,:\, V\,\hookrightarrow\, E$ as 
\[
\iota \,=\, \begin{pmatrix} a\\ b \end{pmatrix},
\] 
with $a\,\in\, \Gamma(V^{-1}\otimes L^{-1}), b\,\in\, \Gamma(V^{-1}\otimes L)$ satisfying 
\[
\overline{\partial}_{V^{-1}\otimes L^{-1}} a + \gamma b \,=\, 0,\quad \overline{\partial}_{V^{-1}\otimes L} b \,=\, 0.
\]
By the second equation, we get a holomorphic section $b\,\in\, H^0(\Sigma,\,V^{-1}\otimes L)$. Since
$b\,\equiv\, 0$ implies $V\,\cong\, L^{-1}$, which has a negative degree, the section
$b$ must be non-trivial. It follows that $$0\,\leq\, \mathrm{degree}(V^{-1}L) \,=\,
\mathrm{degree}(L)-\mathrm{degree}(V) \,=\, 1- \mathrm{degree}(V)\, ,$$
hence, $0\,\leq \,\mathrm{degree}(V) \,\leq\, 1$. If $\mathrm{degree}(V) \,=\, 1$,
then $\mathrm{degree}(V^{-1}\otimes L)\,=\,0$ and $b$ would be an isomorphism, which in turn
would give $V\,\cong \,L$ and so $E\,=\, L^{-1}\oplus L$ holomorphically, violating
the condition $0\,\neq\, \gamma \,\in\, H^1(\Sigma,\,L^{-2})$. Thus, it follows that
$\mathrm{degree}(V) \,=\,0$, consequently, $\mathrm{degree}(V^{-1}\otimes L) \,=\,1$, which implies that
there exists $q\,\in\, \Sigma$ such that $b(q) \,=\,0$. In other words, 
\[
V\,=\, L(p-q)
\] 
holds.
We have $h^0(L(q)) \,=\, h^0(V^{-1}\otimes L) \,=\, 1$ so that $b\,=\,s_q$ is the distinguished
holomorphic section.
By Serre duality, the equation $\overline{\partial}_{V^{-1}L^{-1}} a + \gamma s_q
\,=\, 0$ has a solution if and only if $\gamma\otimes s_q$ annihilates the co-kernel of
$\overline{\partial}_{V^{-1}L^{-1}}$, meaning
\[
\int_\Sigma \langle \gamma\otimes s_q,\, \omega \rangle \,=\, 0, \qquad \forall\ 
\omega \,\in\, H^0(\Sigma,\,K\otimes V\otimes L) \,=\, H^0(\Sigma,\,K\otimes L(2p-q))\, ,
\]
where $\langle\cdot,\,\cdot\rangle$ denotes the pairing between $L^{-2}$ and $L^2$. 

Now since $g\,\geq\, 3$, we may choose a basis $\omega_1,\,\cdots,\, \omega_g\,\in\, H^0(\Sigma,\, K)$ such
that $\omega_1,\,\cdots,\,\omega_{g-1}$ have no common zero. Indeed, we have $\mathbb P(H^0(\Sigma, K)^*)
\,\cong \,\mathbb{CP}^{g-1}$ so that we can just take $\omega_1,\,\cdots,\,\omega_{g-1}$ to span a
hyperplane in the complement of the Kodaira embedding $\Sigma
\,\longrightarrow\, {\mathbb P}(H^0(\Sigma,K)^*)$ and then complete that to a basis. By Riemann-Roch
we have 
\[
h^0(K\otimes L^2) \,=\, 2g-2 + 2 + 1-g \,=\, g+1\, .
\]
Now the bundle $K\otimes L^2$ has degree $$\mathrm{degree}(K\otimes L^2) \,=\,
\mathrm{degree}(K) + 2 \,>\,\mathrm{degree}(K)+1\, ,$$ which implies that it is base-point free.
Take a section $\eta\,\in\, H^0(\Sigma,\,K\otimes L^2)$ such that $\eta(p) \,\neq\, 0$, and consider the hyperplane
\[
U \,=\, \mathrm{span}\{\omega_1s_p^2, \,\cdots,\,\omega_{g-1}s_p^2,\,\eta\}
\,\subset\, H^0(\Sigma,\,K\otimes L^2)\, .
\]
Now pick any $\gamma\,\in\, \Gamma(\overline{K}\otimes L^{-2})$ such that $$U\,=\,
{\rm kernel}\left(\alpha \mapsto \int_\Sigma \langle \gamma \wedge\alpha \rangle\right)\, .$$ Suppose there exists a
$q\,\in\, \Sigma$ such that $$U \,=\, s_q\otimes H^0(\Sigma,\,K\otimes V\otimes L)\,=\, H^0(\Sigma,\,K\otimes L(2p-q))\, .$$
If $p\,\neq \,q$,
then for each $\omega\,\in \,H^0(\Sigma,\,K\otimes L(2p-q))$ we may write $$s_q\omega
\,=\, (\lambda_1\omega_1 + \ldots +\lambda_{g-1}\omega_{g-1})s_p^2 + \mu\eta\, ,$$
with $\mu\,\neq\, 0$. Hence
such a section cannot vanish at $p$, a contradiction. On the other hand, if $p=q$ then $V$ is trivial and we have
$$H^0(\Sigma,\,K\otimes V\otimes L) \,= \,H^0(\Sigma,\,K\otimes L)\, .$$ Then $\alpha s_p$ vanishes to order at least $3$ at $p$, giving
again a contradiction to the choice of $\omega_1,\,\cdots,\, \omega_g$ and $\eta$.
\end{proof}

\begin{proposition}\label{Prop:OddPar}
For every Riemann surface $\Sigma$ of genus $g\,\geq\,2$,
there exists a holomorphic section $s$ of $\mathcal M_{DH}\,\longrightarrow\,\C P^1$
which has non-zero parity and therefore $s$ is not admissible.
\end{proposition}

\begin{proof}
Consider an irreducible, unitary flat ${\rm SL}(2,\C)$-connection $\nabla$ on $\Sigma$ whose underlying
holomorphic structure is wobbly. Let $\Phi$ be a nonzero nilpotent Higgs
field on the holomorphic vector bundle given by $\nabla$.

We claim that the section $s$ defined by
\[\lambda\in\C^*\,\longmapsto\, [\overline{\partial}^\nabla,\,
\lambda\partial^\nabla+\lambda^{-1}\Phi,\, \lambda]\,\in\,\mathcal M_{DH}\]
extends holomorphically as a section to $\lambda\,=\,0,\, \infty$ which is not of parity $0$. This
claim will be proved below. The above section
$s$ corresponds to the family $\nabla^\lambda \,=\, \nabla + \lambda^{-2}\Phi$ of irreducible flat connections.

To prove the above claim, we first note that the fact that the holomorphic connection
defined by $\nabla$ is irreducible and unitary implies that $\partial^\nabla$ defines a stable holomorphic
bundle on the conjugate Riemann surface $\overline\Sigma$. As a consequence, the section $s$ extends to $\infty$.

In order to see that $s$ extends to $\lambda\,=\,0$, consider the kernel line
bundle $L$ of the nilpotent Higgs field $\Phi$, in other words,
\[L_p\,\subset\, {\rm kernel}(\Phi_p)\]
with equality for general $p\,\in\,\Sigma.$ With respect to the unitary splitting
\[V\,=\,L\oplus L^\perp\]
the Higgs field $\Phi$ then has the form 
\[
\Phi \,=\, \begin{pmatrix} 0 & \alpha \\ 0 & 0\end{pmatrix}
\]
for some $\alpha\,\in \,\Omega^{1,0}((L^\perp)^*\otimes L)$. We define a family of gauge transformations
\[\Lambda\,=\,\Lambda(\lambda)\,=\,\begin{pmatrix} 1 & 0 \\ 0 & \lambda \end{pmatrix}\, ,\]
and apply it to a lift $\nabla+\lambda^{-2}\Phi$ of $s$ to obtain
\[(\nabla+\lambda^{-2}\Phi).\Lambda\,=\,\widetilde\nabla+\lambda^{-1}\widetilde{\Phi}
+\lambda\Psi\]
for a diagonal (in general non-flat) connection $\widetilde\nabla$ and off-diagonal
$(1,0)$ and $(0,1)$-forms $\widetilde\Phi$ and $\Psi$.
It is straight-forward to check that $(\overline{\partial}^{\widetilde\nabla},\,\widetilde{\Phi})$
is a stable Higgs bundle. Thus, the section $s$ holomorphically extends to $\lambda\,=\,0$.
We have shown that the section $s$ extends to an irreducible section $s\,:\, \C P^1\, \longrightarrow\, \mathcal M_{DH}$.

As we have already two lifts of $s$ (one for $\lambda\,\neq\,0$ and the other for
$\lambda\,\neq\,\infty$) and these lifts differ by the degree $1$ map
$\Lambda$, the proof of the proposition is complete.
\end{proof}

It would be very interesting to know whether there are any geometrically meaningful 
sections of $\mathcal M_{DH}$ which do not have parity $0$.

\subsection{Real sections} \label{sec:real}

Let $\sigma\,:\, \mathcal M_{DH}\,\longrightarrow\, \mathcal{M}_{DH}$ be an anti-holomorphic 
involution, covering an anti-holomorphic involution $\widetilde\sigma\,:\, \C 
P^1\,\longrightarrow\,\C P^1$, i.e., $\pi\circ \sigma \,=\, \widetilde\sigma\circ \pi$, where 
$\pi\,:\, \mathcal M_{DH}\,\longrightarrow\, \C P^1$ is the projection as usual. The cases of 
interest to us in this paper are $\sigma \,=\, \tau$ and $\sigma \,=\, \rho \,=\, \tau\circ 
N$. The corresponding involutions on $\widetilde\tau,\,\widetilde\rho\, :\,\C 
P^1\,\longrightarrow\, \C P^1$ are then $\widetilde\tau(\lambda) \,=\, -\overline\lambda^{-1}$ 
and $\widetilde\rho(\lambda) \,=\, \overline\lambda^{-1}$ (see \eqref{N}, \eqref{et}, 
\eqref{defC}).

\begin{definition}\label{Def:sigmaReal}
A holomorphic section $s\,:\,\C P^1\,\longrightarrow\,\mathcal M_{DH}$ of the fibration $\pi\,:\, \mathcal M_{DH}\,\longrightarrow\, \C P^1$ is called {\it real with respect to $\sigma$}, or just {\it $\sigma$-real} for short, if 
$s(\lambda)\,=\, \sigma(s(\widetilde\sigma(\lambda)))$ for all $\lambda\,\in\, \mathbb CP^1$.
\end{definition}

\begin{example}\label{Ex:SD_HM_real}\mbox{}
\begin{enumerate}[(i)]
\item A section $s\,:\, \C P^1\,\longrightarrow\, \mathcal M_{DH}$ is $\tau$-real if $s(\lambda)) \,=\, \tau(s(-\overline\lambda^{-1}))$. If we have a lift $\nabla^\lambda$ on $\mathbb C\,\subset\,\mathbb CP^1$ of $s$, this means that for every $\lambda\,\in\,\mathbb C\setminus\{0\}$ there is a gauge
transformation $g(\lambda)$ such that the equation
\begin{equation}\label{realeqsectau}
\nabla^\lambda.g(\lambda)\,=\,\overline{\nabla^{-\overline{\lambda}^{-1}}}
\end{equation}
holds. The calculation \eqref{SD_real} at the end of section \ref{sec:Conj_vs_adj} shows that a solution to the self-duality equations gives rise to a $\tau$-real section.
\item Similarly, a section $s\,:\, \C P^1\,\longrightarrow\, \mathcal M_{DH}$ is $\rho$-real if $s(\lambda) \,=\, \rho(s(\overline\lambda^{-1}))$. In terms of a lift $\nabla^\lambda$ on $\mathbb C\,\subset\,\mathbb CP^1$ of $s$, this
 means that for every $\lambda\,\in\,\mathbb C\setminus\{0\}$ there is a gauge
transformation $g(\lambda)$ such that 
\begin{equation}\label{realeqsecrho}
\nabla^\lambda.g(\lambda)\,=\,\overline{\nabla^{\overline{\lambda}^{-1}}}.
\end{equation}
A computation analogous to the one in \eqref{SD_real} shows that the family of flat connections associated with a minimal surface $f\,:\,\Sigma\, \longrightarrow\,\mathrm{SU}(2)$ defines an admissible $\rho$-real section (see Example \ref{ex:basic}).
\end{enumerate}
\end{example}

\begin{remark}
Besides $\tau$ and $\rho$, we do not expect any other interesting real involutions of $\mathcal M_{DH}$ to exist in general. To see this, 
recall that the fibers of $\pi$ over $\lambda\,=\,0,\,\infty$ are special (and not biholomorphic to the fibers over $\lambda\neq0$), compare 
also with \cite{BGHL}. On the other hand, if $\Sigma$ itself admits a real involution $\mathfrak i$ , the following construction is possible: 
By pull-back with $\mathfrak i$ we obtain a (holomorphic) involution
\[\mathfrak i\, \colon\, \mathcal M_{DH}\,\longrightarrow\, \mathcal M_{DH}\] covering the map
$\lambda\,\longmapsto\,\tfrac{1}{\lambda}.$ As $\mathfrak i$ and $\rho$ commute, we obtain a real involution $\mathfrak i\circ\rho$ covering 
$\lambda\,\longmapsto\,\overline{\lambda}.$ As opposed to the cases of $\tau$ and $\rho$, $\mathfrak i\circ\rho$-real sections are not directly related to 
harmonic maps. Instead, they give rise to a generalized reflection symmetry of the family of $\lambda$-connections. In the special case when 
the section is $\mathfrak i\circ\rho$-real and $\tau$-real or $\rho$-real, the corresponding (equivariant) harmonic map admits a reflection 
symmetry along a totally geodesic subspace of $H^3$ respectively $S^3,$ for a visualization in the latter case see \cite[Figures 1,3 $\&$ 
4]{HeHeSch}.
\end{remark}

In the following, we will only work with irreducible $\sigma$-real sections $s$ of $\mathcal 
M_{DH}$ of parity $0$, where $\sigma \,\in\, \{\tau,\, \rho\}$. In this case, we can choose the family of gauge transformations $g(\lambda)$ in \eqref{realeqsectau} and \eqref{realeqsecrho} to depend holomorphically on $\lambda$ and may assume that it takes values in $\SL(2,\mathbb C)$. By irreducibility, the holomorphic family $g(\lambda)$ is then uniquely determined up to a sign.

\begin{lemma}\label{lemma-pos-quat}
Let $\sigma\,\in\,\{\tau,\,\rho\}$ and consider an irreducible $\sigma$-real section $s$ of parity $0$. Let $\nabla^\lambda$ be a lift of $s$ over $\C$ and let $g(\lambda)$, $\lambda\in\C\setminus\{0\}$, be a holomorphic family of $\SL(2,\mathbb C)$-valued gauge transformations such that 
\begin{equation}\label{realeqsecsigma}
\nabla^\lambda.g(\lambda)\,=\,\overline{\nabla^{\widetilde{\sigma}(\lambda)}}.
\end{equation}
Then,
\begin{equation}\label{posneg}
g(\lambda)\overline{g(\widetilde\sigma(\lambda))}\,=\,\pm{\rm Id}
\end{equation}
with the above sign being independent of the choice of the lift. 
\end{lemma}

\begin{proof}
We have for all $\lambda\,\in\,\mathbb C\setminus\{0\}$,
\begin{equation}
\begin{split}
\nabla^\lambda&\,=\,\overline{\overline{\nabla^\lambda}}\,=\,\overline{\nabla^{\widetilde\sigma(\lambda)}. g(\widetilde\sigma(\lambda))}\\
&\,=\,\overline{\nabla^{\widetilde\sigma(\lambda)}}.\overline{ g(\widetilde\sigma(\lambda))}\,=\,\nabla^\lambda .g(\lambda) \overline{g(\widetilde\sigma(\lambda))}.
\end{split}
\end{equation}
As $g$ is $\SL(2,\mathbb C)$-valued, and $\nabla^\lambda$ is irreducible for all
$\lambda$, it follows that
\[g(\lambda) \overline{g(\widetilde{\sigma}(\lambda))}\,=\, \pm\text{Id}\, .\]
If $\widetilde{\nabla}^\lambda$ is another lift, then
there is a $\SL(2,\C)$-valued gauge transformation $h$ which depends holomorphically
on $\lambda$, extends to $\lambda\,=\,0$ and satisfies the equation
 \[\widetilde{\nabla}^\lambda\,=\,\nabla^\lambda.h(\lambda)\, .\]
 Thus, we have for $\lambda\in\C\setminus\{0\}$
 \[\widetilde{\nabla}^\lambda.h^{-1}(\lambda )g(\lambda) \overline{h(\widetilde\sigma(\lambda))}\,=\,\overline{\widetilde{\nabla}^{\widetilde\sigma(\lambda)}}\]
 and 
 \[{\widetilde g}(\lambda)\,=\,h^{-1}(\lambda )g(\lambda) \overline{h(\widetilde\sigma(\lambda))}\]
 clearly satisfies the equation
 \[{\widetilde g}(\lambda)\overline{{\widetilde g}(\widetilde\sigma(\lambda))}\,=\, g(\lambda)\overline{g(\widetilde\sigma(\lambda))}.\]
This completes the proof of the lemma.
\end{proof}

The above lemma justifies the following definition.

\begin{definition}\label{Def:RealPosNeg}
Let $\sigma\,\in\,\{\tau,\,\rho\}$ and consider an irreducible $\sigma$-real section $s$ of parity $0$. Let $\nabla^\lambda$ be a lift of $s$ over $\C$ and let $g(\lambda)$, $\lambda\,\in\,\C\setminus\{0\}$, be a holomorphic family of $\SL(2,\mathbb C)$-valued gauge transformations such that 
\eqref{realeqsecsigma} holds. Then $s$ is called {\it $\sigma$-positive} if $g(\lambda)\overline{g(\widetilde\sigma(\lambda))}\,=\,{\rm Id}$ and {\it $\sigma$-negative} if $g(\lambda)\overline{g(\widetilde\sigma(\lambda))}\,=\,-{\rm Id}$.
\end{definition}
In the next two sections we shall study $\tau$-real and $\rho$-real sections more closely and link them to various kinds of harmonic maps.

\section{Real structure covering $\lambda\longmapsto-\overline{\lambda}^{-1}$ and real sections}\label{sec:realm}

We consider sections that are real with respect to the anti-holomorphic involution
\[\tau\,\colon\, \mathcal M_{DH}\,\longrightarrow\, \mathcal M_{DH} \]
defined in \eqref{et}
covering the anti-holomorphic involution
\[\lambda\,\longmapsto\, -\overline{\lambda}^{-1}\] of $\mathbb CP^1$, see Definition \ref{Def:sigmaReal} and Example \ref{Ex:SD_HM_real}.

We have seen in Example \ref{Ex:SD_HM_real} that a solution to the self-duality equations gives rise to a $\tau$-real section. This is in fact a general feature of the twistor construction in hyper-K\"ahler geometry: Points in the hyper-K\"ahler manifold give sections that are real with respect to the natural anti-holomorphic involution on the twistor space. In the literature these are sometimes called {\it twistor lines}. Here we shall follow Simpson's terminology \cite{Si} and call them \emph{preferred}. The precise definition is as follows.

\begin{definition}
We call a $\tau$-real section of $\mathcal M_{DH}\,\longrightarrow\,\C P^1$
a {\it preferred} section if it is given by the family of integrable $\lambda$-connections
associated with some solution to Hitchin's self-duality equations.
\end{definition}

\begin{question}\label{q1}
Simpson raised the question whether every $\tau$-real section of the Deligne-Simpson twistor
space $\mathcal M_{DH}$ is in fact a preferred section, i.e., whether it comes from a
solution of the self-duality equations \cite[\S~4, p.~235, Question]{Si}.
\end{question}

\begin{proposition}\label{SDneg}
Let $s$ be an irreducible preferred section. Then $s$ is $\tau$-negative.
\end{proposition}

\begin{proof}
Let $(\nabla,\,\Phi)$ be a $\SL(2,\C)$-solution of the self-duality equations with 
respect to the standard Hermitian metric $(-,-)$ on the trivial vector bundle $E\,=\, \Sigma\times{\mathbb
C}^2$. Note that any solution of the self-duality equations for some 
Hermitian metric can be gauged into this form; see Section \ref{sec:SDaM}.

In \eqref{SD_real} we have seen the identity
\[\overline{\nabla^\lambda}.\delta\,=\,\nabla^{-\overline{\lambda}^{-1}}.\]
As $\delta^2\,=\,-\text{Id}$, every solution to the self-duality equations yields a $\tau$-negative section of the Deligne-Hitchin moduli space. 
\end{proof}

In view of Proposition \ref{SDneg}, the 
following theorem shows that Question \ref{q1} has a
negative answer.

\begin{theorem}\label{countertheorem}
There exist irreducible, admissible, $\tau$-positive sections of $\mathcal M_{DH}.$
\end{theorem}

Before starting the proof, let us briefly describe the main ingredients.
The counterexamples are given by generalizations of non-conformal harmonic maps to the round 2-sphere, 
which are called harmonic sections in \cite{Hi}.
They correspond to
solutions of the signed self-duality equations
\begin{equation}\begin{split}\label{harmonics2}
\overline{\partial}^\nabla\Phi&=0\\
F^\nabla-[\Phi\wedge\Phi^*]&=0
\end{split}
\end{equation}
for a special unitary connection $\nabla$ and a Higgs field $\Phi.$ The additional condition we need to
impose is that there exists a gauge transformation $g$ such that
\begin{equation}\label{is2}
g^2\,=\,-\text{Id},\;\;\; \nabla.g\,=\,\nabla \;\;\; \text{ and } \,\,\;g^{-1}\Phi g
\,=\,-\Phi\, .
\end{equation} 
The above additional property means that the (equivariant) harmonic map is contained
in a 2-sphere; see \S~1 of \cite{Hi}. If the conditions in \eqref{is2} are not satisfied,
then we would obtain a harmonic map into the 3-sphere $\SU(2)$ with monodromy.

The existence of solutions of \eqref{harmonics2} and \eqref{is2} can be seen as follows. The Gauss map of a 
minimal surface in $S^3$ is a non-conformal harmonic map into $S^2.$ Minimal surfaces exist for every genus 
\cite{La}. However, the family of flat connections associated to such a minimal surface does \emph{not} define an 
irreducible section of $\mathcal M_{DH}$ as there exist two points $\lambda_1,\,\lambda_2\,\in\, S^1$ such that the 
corresponding flat connections $\nabla^{\lambda_1}$ and $\nabla^{\lambda_2}$ are trivial and in particular
these two flat connections are reducible.

To obtain an irreducible section of the above kind, we use the existence of wobbly 
Higgs bundles $(E,\,\overline\partial,\,\Phi)$, where $(E,\,\overline\partial)$ is stable of degree zero with trivial determinant, and $\Phi$ is nilpotent (see Proposition \ref{StableNilpotent}).

\begin{proof}[Proof of Theorem \ref{countertheorem}]
Using Proposition \ref{StableNilpotent} and the results of \cite{Roe} we obtain a solution $(\widetilde{\nabla},\,\widetilde{\Phi})$ of
\eqref{harmonics2} such that
\begin{itemize}
\item $\widetilde\nabla$ is irreducible with small curvature, and

\item the Higgs field $\widetilde\Phi$ is nilpotent and small.
\end{itemize}
(An independent proof of the above mentioned result of \cite{Roe} is given in
the proof of Theorem \ref{twistordes} below.)
We obtain the associated family 
 \[\lambda\longmapsto\widetilde{\nabla}^\lambda \,=\, \widetilde\nabla + \lambda^{-1}\widetilde{\Phi} -\lambda\widetilde{\Phi}^* \]
of flat connections (see Example \ref{ex:basic}).

We claim that the corresponding section $\widetilde s$ of $\mathcal M_{DH}$ is irreducible (for $\widetilde{\Phi}$ small enough).

In order to prove that $\widetilde s$ is irreducible, we first note that irreducibility is an open condition. As the constant section defined 
by the flat unitary connection corresponding to $\overline\partial^{\widetilde{\nabla}}$ (via the Narasimhan-Seshadri theorem)
 is irreducible, any section nearby is irreducible as well because the domain $\mathbb CP^1$ of the section is compact. As $\widetilde s$ is a 
small deformation of the constant section if $\widetilde{\Phi}$ is small (compare also with the first part of the proof of Theorem 
\ref{twistordes} below), the section $\widetilde s$ is in fact irreducible. This proves the claim.

There is a transformation on $\widetilde{\nabla}^\lambda$, corresponding to the Gauss map of the equivariant conformal harmonic map given by 
$(\widetilde{\nabla},\,\widetilde\Phi)$, which yields a solution of \eqref{harmonics2} and \eqref{is2} as follows: As $\widetilde\Phi$ is 
nilpotent, it has a holomorphic kernel bundle $L.$ Let
$$
\underline \C^2\,=\,L\oplus L^\perp
$$
be the orthogonal decomposition, so $L^\perp$ is the orthogonal complement of $L$.
Consider the new family of flat connections defined by
\begin{equation}\label{defnablam}
\nabla^\lambda=\widetilde{\nabla}^{\lambda^2}.h(\lambda),\end{equation}
where 
\[h(\lambda)=\begin{pmatrix} \sqrt{\lambda} &0\\0 &\frac{1}{\sqrt{\lambda}} \end{pmatrix}\]
with respect to the above decomposition
$\underline \C^2\,=\,L\oplus L^\perp.$ Note that $h$ is \emph{not} an element of the loop group, as it is not single-valued on $\C\setminus \{0\}$. Since the ambiguity in the definition of $h$
involves just a sign, the resulting family of connections is indeed well-defined. Also, $\nabla^\lambda$ as defined in \eqref{defnablam} is of the form
\[\nabla^\lambda=\nabla+\lambda^{-1}\Phi-\lambda\Phi^*,\]
 where $\nabla$ and $\Phi$ are defined implicitly by this equation. 
In fact, we can write down $\nabla$, $\Phi$ and $\Phi^*$ more explicitly.
With respect to the above orthogonal decomposition
$\underline \C^2\,=\,L\oplus L^\perp$, where $L$ is the kernel bundle of $\widetilde{\Phi}$,
we have an induced decomposition of the unitary connection
\[\widetilde{\nabla}=\begin{pmatrix}\nabla^L& \gamma \\ -\gamma^* & \nabla^{L^*}\end{pmatrix}.\]
Here, $\nabla^L$ and $\nabla^{L^*}$ are ${\rm U}(1)$ connections, while $\gamma\,\in\,\Gamma(\Sigma,\,\overline{K}\otimes L^2)$ and 
$-\gamma^*\,\in \,\Gamma(\Sigma,\,K\otimes L^{-2})$ are $C^\infty$ sections.
 Recall that $\widetilde{\Phi}$ has the form
\[\widetilde{\Phi}=\begin{pmatrix} 0 &\alpha\\0&0\end{pmatrix}\]
for a non-vanishing holomorphic section $\alpha\,\in\, H^0(\Sigma,\,K\otimes L^2).$ Then, the connection $\nabla$ is given by
\begin{equation}\label{nabLL*}
\nabla=\begin{pmatrix} \nabla^L&0\\0&\nabla^{L^*}\end{pmatrix}\, ,\end{equation}
while the Higgs field is 
\begin{equation}\label{higgsLL*}
\Phi=\begin{pmatrix} 0& \alpha\\ -\gamma^*&0\end{pmatrix}\end{equation}
and its adjoint has the expression
\[\Phi^*=\begin{pmatrix} 0& -\gamma\\ \alpha^*&0\end{pmatrix}.\]
 It is then not difficult to see that the family $\nabla^\lambda$ gives rise to a
solution of \eqref{harmonics2} and \eqref{is2} by taking 
\[g\,=\,\begin{pmatrix}\sqrt{-1} &0\\0 &-\sqrt{-1} \end{pmatrix}\]
with respect to $\underline \C^2=L\oplus L^\perp.$

We claim that this family yields an irreducible section $s$ (with $\nabla^\lambda$ as an admissible lift) which is $\tau$-positive.

We first prove the irreducibility of the section $s$.
Note that for $\lambda\,\neq\,0$, the irreducibility of a $\lambda$-connection $(\overline{\partial},\,D,\, \lambda)$ is equivalent to 
the irreducibility of the corresponding
flat connection $\overline{\partial}+\tfrac{1}{\lambda}D.$ Since $\widetilde s$ is an irreducible section, the connections
$\widetilde{\nabla}^\lambda$ are irreducible for all $\lambda\,\in\,\C^*.$ Therefore, also the connections $\nabla^\lambda$
are irreducible for all $\lambda\in\C^*.$ In order to prove that $s$ is an irreducible section it suffices to prove that
\[(\overline{\partial}^\nabla,\, \Phi)\]
is an irreducible (i.e., stable) Higgs bundle (note that the corresponding statement for $\lambda=\infty$ follows from the reality of $s$).
The stability of $(\overline{\partial}^\nabla,\, \Phi)$
can be easily deduced from \eqref{nabLL*} and \eqref{higgsLL*}: As $\alpha\,\neq\,0$, the line subbundle $L^*$ is not invariant under 
$\Phi$. Hence, every $\Phi$-invariant line subbundle $V$ must admit a non-trivial holomorphic line bundle homomorphism 
$$V\,\longrightarrow\, {\underline{\C}^2}/L^*\,=\, E/L^*\,=\,L,$$ and therefore we have $$\deg(V)\,\leq\,\deg(L)\,<\,0$$ as claimed.

In order to see that $s$ is $\tau$-positive, we use the above gauge transformation $g$ and the gauge transformation $\delta$ introduced
in Section \eqref{sec:Conj_vs_adj} and obtain:
\begin{equation}\label{gdelta}
\begin{split}
\overline{\nabla^{-\overline{\lambda}^{-1}}}&=\overline{\nabla}-\lambda \overline{\Phi}+\lambda^{-1}\overline{\Phi}^*\\
&=\nabla.g\delta+\lambda^{-1} (g\delta)^{-1}\Phi (g\delta)-\lambda (g\delta)^{-1}\Phi^*(g\delta)\\ &=\nabla^\lambda.g\delta.
\end{split}
\end{equation}
This shows that $s$ is $\tau$-real.
A short computation (after taking care about the fact that the matrix representations of $g$ and $\delta$ are with respect to the splittings $L\oplus L^\perp$ and $\C\oplus\C,$ respectively) yields
\[g\delta\overline{g}\overline{\delta}\,=\,\text{Id}\, ,\]
so $s$ is $\tau$-positive as claimed.
\end{proof}

\begin{remark}\mbox{}
\begin{enumerate}[(i)]
\item From our construction, we were unable to determine the normal bundle of the resulting admissible $\tau$-positive section. It is a general feature of the twistor construction in hyper-K\"ahler geometry that the normal bundle of a preferred section is isomorphic to
${\mathcal O}_{\C P^1}(1)^{\oplus d}$, where $d \,=\, \dim_\C(\mathcal M_{SD})$. It follows from the deformation argument in the proof of Theorem \ref{twistordes} that this also holds for the normal bundle of the $\rho$-real section given by $\widetilde\nabla^\lambda$. In order to calculate the normal bundle of the $\tau$-positive section $\nabla^\lambda$ constructed in the proof above, it would be necessary to understand how the normal bundle is affected by conjugation by the family $h(\lambda)$.

\item In a recent preprint \cite{HeHe}, L.~Heller and the second author construct 
$\tau$-negative sections of the Deligne-Hitchin moduli space which are not admissible. 
\end{enumerate}
\end{remark}

\begin{theorem}\label{mainth}
Let $s$ be an irreducible, admissible, $\tau$-real section. Then $s$ is
\begin{itemize}
\item either induced by a solution of the self-duality equations and hence a preferred section,

\item or it is $\tau$-positive.
\end{itemize}
\end{theorem}

\begin{proof}
Let us assume that the section $s$ is $\tau$-negative. Consider an admissible lift 
\[\nabla^\lambda\,=\,\nabla+\lambda^{-1}\Phi+\lambda\Psi\]
with $\Phi$ of type $(1,0)$ and $\Psi$ of type $(0,1)$.
By definition of admissibility (Definition \ref{admissible}) and Lemma \ref{lemma-pos-quat}, we have a family of $\SL(2,\C)$-gauge transformations
$g(\lambda)$ with
\begin{equation}\label{geq}
\overline{\nabla^{-\overline{\lambda}^{-1}}}=\nabla^\lambda.g(\lambda)\end{equation}
and
\begin{equation}\label{ggb=-id}-\overline{g(-\overline{\lambda}^{-1})}=g(\lambda)^{-1}.\end{equation}

We claim that $g$ is constant.

In order to prove that $g$ is constant, observe that we can extend $g$ holomorphically to $\lambda=0$ as follows:

We define a holomorphic line bundle $\mathcal L\,\longrightarrow\,\C P^1$ whose fibers over $\lambda\in\C^*$ are parallel homomorphism with 
respect to the $\lambda$-connection on $\text{End}(E)\,=\, E\otimes E^*$ given by 
$\nabla^\lambda$ and the dual of $\overline{\nabla^{-\overline{\lambda}^{-1}}}.$
Observe that the fiber $\mathcal L$ over $\lambda\,\in\,\C^*$ can be described as the space of global homomorphisms 
between the $\lambda$-connections
\[(\overline{\partial}^\nabla+\lambda\Psi,\, \lambda\partial^\nabla+\Phi,\, \lambda)\]
and
\[(\overline{\partial}^{\overline\nabla}-\lambda\overline\Phi,\, \lambda\partial^{\overline\nabla}-\overline\Psi, \, \lambda).\]
Since $(\overline{\partial}^\nabla+\lambda\Psi,\, \lambda\partial^\nabla+\Phi, \, \lambda)$ is irreducible, such a nonzero homomorphism must be
an isomorphism because its kernel, if it were nonzero, would contradict irreducibility. Consequently, any two such nonzero homomorphism would produce
an automorphism of $(\overline{\partial}^\nabla+\lambda\Psi,\, \lambda\partial^\nabla+\Phi,\, \lambda)$, hence they differ by a constant scalar
multiplication, again by irreducibility. Hence $\mathcal L$ is indeed a line bundle over $\C^*$.

Therefore, the line bundle $\mathcal L$ extends to $\lambda\,=\,0$ in such a way that the fiber over $0$ is given by the isomorphisms
between the stable Higgs pairs
\[(\overline{\partial}^\nabla,\, \Phi)\] and\[ (\overline{\partial}^{\overline\nabla},\, -\overline\Psi)\]
(together with the $0$-endomorphism). Of course, the line bundle $\mathcal L$ also extends holomorphically to $\lambda=\infty$
in an analogous manner.

We take a meromorphic section $g$ of $\mathcal L$ whose only zeros or poles are at $\lambda=\infty$. As the determinant of $g$ cannot be zero at some $\lambda\in \C\subset\C P^1$ (the connections are irreducible and the
Higgs pairs are stable, so every non-zero point in a fiber is invertible) it is constant and without loss of generality it can be
taken to be $1$.
Hence, we have the existence of a polynomial $g$ satisfying
\eqref{geq} and \eqref{ggb=-id}. But the inverse of a polynomial $\SL(2,\C)$-matrix is polynomial again, so we obtain that $g$ is constant in $\lambda$ by
looking at \eqref{ggb=-id}. This proves the claim. 

Next, from \[\overline{g}g\,=\,-\text{Id}\]
one can deduce the existence of
a smooth gauge $h\,\colon\,\Sigma\,\longrightarrow\, \SL(2,\C)$ such that
\begin{equation}\label{eqgh}g\,=\,-h\delta \overline{h}^{-1}\, ,\end{equation}
and \[\nabla^\lambda.h\] is a preferred
section, meaning, it comes from a solution of the self-duality equations.
\end{proof}

\begin{remark}
The statement of Theorem \ref{mainth} can be translated into Simpson's language of mixed twistor
structures \cite{Si3}; see also \cite{HS}, \cite{Sc}.
\end{remark}

In the case of a $\tau$-positive admissible section we can proceed as in the proof of Theorem \ref{mainth} and obtain a family of flat connections
\[\nabla^\lambda\,=\, \nabla+\lambda^{-1}\Phi+\lambda\Psi\, ,\]
and a map $g\,\colon\,\Sigma\,\longrightarrow\,\SL(2,\C)$,
satisfying the following two conditions:
\[\overline{\nabla^{-\overline{\lambda}^{-1}}}\,=\, \nabla^\lambda.g\]
and \[g\overline{g}\,=\, \text{Id}.\]
The following lemma is likely to be well-known, but we include a proof of it for completeness.

\begin{lemma}\label{circle_lemma}
The space of oriented circles on $S^2 \,=\, \C P^1$ is given
\[C\,=\,\{g\,\in\, \SL(2,\C)\,\mid\, g\overline{g}\,=\,{\rm Id}\}\,.\]
 It is isomorphic
to the space $\SL(2,\C)/\SL(2,\R)$ and the action
of $h\,\in\, SL(2,\C)$ on $g\,\in\, C$ is
\[g.h\,=\,h^{-1}g\overline{h}\, .\]
\end{lemma} 

\begin{proof} 
An oriented circle $\gamma$ in $S^2$ may be written as $\gamma= P\cap S^2$, where $P\subset \R^3$ is an oriented plane. Such a plane is uniquely determined by its normal vector $v\in S^2$ together with a constant $c\in\R$:
\[
P\,=\, \{x\,\in\,\R^3\ \mid \ \langle v,x\rangle \,=\, c\}\, .
\]
The constant $c$ must satisfy $|c|<1$, so that $P$ and $S^2$ intersect in a circle. In other words, $C\,=\,S^2\times (-1,\,1)$, which is homotopy equivalent to $S^2$. The map $(v,c)\longmapsto \frac{1}{1-c^2}(v,c)$ establishes a diffeomorphism 
\[C\,\cong\, \{y\,\in\, \R^{3,1}\ \mid \ \|y\|_{3,1}^2 \,=\, 1\}\, .\]
We consider $\R^{3,1}$ as the space $H$ of complex skew-Hermitian $2\times 2$-matrices 
\[ 
H \,= \,\left\{A\,=\,\left(\begin{array}{cc}\sqrt{-1}\lambda & \overline{w} \\ -w &
\sqrt{-1}\mu \end{array}\right)\ \mid \ \lambda,\mu\,\in\,\R,\ w\,\in\,\C\right\} 
\]
equipped with the quadratic form given by the determinant: 
\begin{equation}\label{qf}
A\,=\,\left(\begin{array}{cc}\sqrt{-1}\lambda & \overline{w} \\ -w &
\sqrt{-1}\mu \end{array}\right)\, \longmapsto\, 
\det(A) \,=\, |w|^2-\lambda\mu\, . 
\end{equation}
{}From this identification it follows that
\[ C \,\cong\, H\cap \SL(2,\C).\]
Define a right-action of $\SL(2,\C)$ on $H$ by 
\[
h\, \longmapsto\, \{A\,\longmapsto\, h^TA\overline{h}\}\, .
\]
This action preserves the quadratic form in \eqref{qf} since $\det h \,=\,1$.
Consequently, we get an action of $\SL(2,\C)$ on $C$. In fact, this is a way to realize the double cover $\SL(2,\C)
\,\longrightarrow\, \SO_+(3,1)$. In particular, the action of $\SL(2,\C)$ on $C$ is
transitive; indeed, this follows using the fact that the action of
$\SO_+(3,1)$ on the unit sphere in $\R^{3,1}$ is transitive. Now consider the matrix
$\delta \,=\, \left(\begin{array}{cc} 0 & -1 \\ 1 & 0\end{array}\right)$. Multiplication
by $\delta$ on the left gives an identification 
\[
C\,\longrightarrow\, \{g\,\in\, \SL(2,\C)\,\mid\, g\overline{g}\,=\,\text{Id}\}\, .
\]
Indeed, the condition $g\overline{g} = \text{Id}$ implies that $g$ is of the form 
\[ 
g \,=\, \left(\begin{array}{cc}w & -\sqrt{-1}
\mu \\ \sqrt{-1}\lambda & \overline{w}\end{array}\right) = \delta \left(\begin{array}{cc}\sqrt{-1}\lambda & \overline{w} \\ -w & \sqrt{-1}\mu \end{array}\right).
\]
Now if $h\,\in\,\SL(2,\C)$, then using $h\delta h^T \,= \,\delta$, i.e., $\delta h^T
\,=\, h^{-1}\delta$ we see that
\[ 
\delta h^TA\overline{h} \,=\, h^{-1}\delta A\overline{h}\, ,
\]
as desired. It is now clear that $\SL(2,\R)$ is the stabilizer of the point $\text{Id} \in C$, and the identification $C\cong\SL(2,\C)/\SL(2,\R)$ follows.
\end{proof}

We see that a $\tau$-positive admissible section gives rise to a map $g\,\colon\,\Sigma\,\longrightarrow\, C$. 
Every such map has a degree as $C\,
=\,S^2\times (-1,\,1)$.
\begin{lemma}\label{ghhmb}
A map $g\,\colon\, \Sigma\,\longrightarrow\, C$ is given by
\[
g\,=\, h\overline{h}^{-1}\, ,
\] 
for some $h\,\colon\, \Sigma\,\longrightarrow\, \SL(2,\C)$
if and only if $g$ is of degree $0.$
\end{lemma}

\begin{proof}
Any map $h\,\colon\, \Sigma\,\longrightarrow\,\SL(2,\C)$ is homotopic to the constant map which maps to the identity element. Therefore, the action $h\,\longmapsto 
\,g.h\,=\,h^{-1}g\overline{h}$ cannot change the degree
of the map $g$. On the other hand, if $g$ has degree 0, it is not too hard to construct such a map $h$ by hand.
\end{proof}

We would like to put any
 $\tau$-positive admissible section $s$ of $\mathcal M_{DH}$
into a real normal form
\[\nabla+\lambda^{-1}\Phi-\lambda\overline{\Phi}\, ,\]
where $\overline{\nabla}\,=\,\nabla,$ analogously to the associated family of flat connections of a $\tau$-negative admissible (i.e. preferred) section.
Lemma \ref{ghhmb} shows that there is a topological obstruction which prevents us from doing so in the general case.

\section{Real structure covering $\lambda\longmapsto\overline{\lambda}^{-1}$ and real sections}\label{sec:realp}

We have already seen in Examples \ref{ex:basic} and \ref{Ex:SD_HM_real} that harmonic maps $f\,\colon\, \Sigma
\,\longrightarrow\, S^3\,=\, {\rm SU}(2)$ give rise to sections of $\mathcal M_{DH}(\Sigma)\,\longrightarrow\,\C P^1$ which are real with
respect to the natural lift 
$$\rho \,= \,\tau\circ N\,\colon\, \mathcal M_{DH}\,\longrightarrow\, \mathcal M_{DH}$$ 
defined in Section \ref{sec:Involutions} of the involution $\lambda\,\longmapsto\,\overline{\lambda}^{-1}$. 

With respect to a lift $\nabla^\lambda$ on $\mathbb C\subset\mathbb CP^1$ of $s$, the property of $s$ being $\rho$-real
exactly means that for 
every $\lambda\,\in\,\mathbb C\setminus\{0\}$ there is a gauge transformation $g(\lambda)$ such that
\begin{equation}\label{realeqsecrho4}
\nabla^\lambda.g(\lambda)\,=\,\overline{\nabla^{\overline{\lambda}^{-1}}}\, .
\end{equation}
As in the case of the real involution $\tau$ in the previous section, we only consider irreducible sections $s$ of $\mathcal M_{DH}$ of parity $0$ and holomorphic families $g(\lambda)$ of gauge transformations with values in $\SL(2,\mathbb C)$. Again, the family of gauge transformation $g(\lambda)$ in \eqref{realeqsecrho4} is then unique up to sign.

\begin{proposition}
Any $\rho$-negative section is obtained from a solution to the harmonic map equations and admissible.
\end{proposition}
\begin{proof} 
This proposition follows from
the fact that the $SU(2)$ loop group Iwasawa decomposition is globally defined on the loop group \cite{PrSe}, and the section $s$ is thus automatically admissible, see \cite[Theorem 6]{He3}.
\end{proof}

If $s$ is admissible and $\rho$-positive, then there exists a lift
\[\nabla^\lambda\,=\,\lambda^{-1}\Phi+\nabla+\lambda\Psi\]
of $s$, and once again we can deduce that $g$ is constant in $\lambda$. As in the discussion before
Lemma \ref{circle_lemma} we see that $g$ is a map into the space $C$ of oriented circles in $\C P^1,$
and as $C$ retracts to $S^2$ this map has a degree.

\begin{definition}
The {\it $g$-degree} of a $\rho$-positive admissible section $s$ is defined to be the degree of the map $g$ as in \eqref{realeqsecrho}.
\end{definition}

It is not hard to see that the $g$-degree is really an invariant of $s$ and does not depend on the choice of an admissible lift.
The relevance of the notion of the $g$-degree comes from its relation to the Toledo invariant:

\begin{theorem}
Let $[(\overline{\partial},\,\Phi)]$ be the gauge orbit of a stable Higgs pair which
is invariant under the involution $N$ \eqref{N}, with $\Phi\,\neq\,0.$ Then, the corresponding preferred section $s$
of $\mathcal M_{DH}$ is also a positive $\rho$-real section, and its Toledo invariant equals to its $g$-degree.
\end{theorem}

\begin{proof}
The first part of the theorem, which says that $s$ is a positive $\rho$-real section, is totally
analogous to the second part of the proof of Theorem \ref{countertheorem}, using the fact from
\cite{Hi1} that $\overline\partial$ must be reducible to the direct sum of line bundles $L\oplus \widetilde{L}$ (where $\text{degree}(L)\,\geq\,0$). The Toledo invariant (first defined by Goldman in the rank 2 case) is then given by the degree
of the destabilizing line subbundle $L$. 

We claim that the map $g\,\colon\,\Sigma\,\longrightarrow\, C$ 
takes values in the space of oriented great circles in $\C P^1$.

The above claim can be easily deduced from the description in Lemma \ref{circle_lemma} as follows: 
we have, similarly to \eqref{gdelta} in the proof of Theorem \ref{countertheorem}, that
\[g\,=\,\widetilde{g} \delta\]
where $\widetilde g$ acts by multiplication with $\pm \sqrt{-1}$ on $L$ and $L^*$,
respectively. Then, the map $g$ corresponds to 
\[\delta^{-1}\widetilde{g}\delta\]
in the space $H$ of skew-Hermitian $2\times2$ matrices of determinant 1 via the
identification in Lemma \ref{circle_lemma}. Moreover, this $\delta^{-1}\widetilde{g}\delta$ has trace $0$ which implies via
the identification $(v,\,c)\,\longmapsto\, \frac{1}{1-c^2}(v,\,c)$ that $g$ maps to the space of great circle. This proves the
claim.

Finally, the degree of the map $g$ is then given by the degree of its
eigen-line bundle $L$ for the eigenvalue $\sqrt{-1}$; see e.g. \cite[Section 2.1]{He} and the references therein.
\end{proof}

\subsection{Description of hyper-symplectic structure for harmonic sections}\label{twist-harm}

In \cite{Roe}, the third author has shown that an open subset $\mathcal R$ of the moduli space of gauge theoretic solutions of the harmonic 
map equations \eqref{harmonics2} admits the structure of a hyper-symplectic manifold. Recall that a hyper-symplectic structure on a manifold 
$M^{4k}$ is given by a quadruple $(g,\,I,\,S,\,T)$, where $g$ is a pseudo-Riemannian metric of signature $(2k,2k)$ and $I,\, S,\, T\,\in\, 
\Gamma(\mathrm{End}(TM))$ are parallel skew-adjoint endomorphisms satisfying $$\mathrm{id}\,=\, - I^2 \,=\, S^2 \,=\, T^2
\ \ \text{ and }\ \ IS \,=\, T \,=\, -SI\, ;$$ see 
\cite{Hi4}, \cite{DancerSwann}. While \cite{Cortes} discusses the general twistor spaces for paraquaternionic K\"ahler manifolds (including 
hyper-symplectic manifolds), and \cite{BRR} describe twistor theory for hyper-symplectic manifolds, we give an interesting twistorial 
description of the hyper-symplectic structure using the Deligne-Hitchin setup to re-obtain some of the results of \cite{Roe}. In particular, 
we obtain the hyper-symplectic analogue of Theorem 3.3 of \cite{HKLR}.

\begin{theorem}\label{twistordes}
Let $\Sigma$ be a compact Riemann surface of genus $g$ and $\mathcal M_{DH}$ its Deligne-Hitchin twistor space
equipped with the real involution $\rho$ covering $\lambda\,\longmapsto\,\overline{\lambda}^{-1}$.
Then, there exists an open subset $\mathcal R$ in the space of $\rho$-real sections such that
\begin{enumerate}
\item $\mathcal R$ contains the constant sections corresponding to irreducible flat unitary connections;
\item the normal bundle $\mathcal N_s$ of every section $s\in\mathcal R$ is $\mathbb C^d\otimes{\mathcal O}_{\C P^1}(1)$, where $d\,=\,6g-6;$
\item the induced real structures of the normal bundles 
have the property that non-trivial real holomorphic sections are not vanishing over $\{|\lambda|\neq 1\}$.
\end{enumerate}
The space $\mathcal R$ is a manifold whose tangent space at $s$ is given by the real holomorphic sections of
$\mathcal N_s.$ Moreover, $\mathcal R$ is equipped with a hyper-symplectic structure which is uniquely determined by the above data and the twistor space structure on $\mathcal M_{DH}$.\end{theorem}
\begin{proof}

Every irreducible flat unitary connection $\nabla$ determines a preferred section \[s(\lambda) \,=
\,(\overline{\partial}^\nabla,\lambda\partial^\nabla, \, \lambda).\] The section $s$ is constant in the sense that the corresponding family of flat connections is just $\nabla^\lambda = \nabla$, and it is preferred as $(\nabla,0)$ is a solution of the self-duality equations. The normal bundle of $s$ is thus $\mathbb C^d\otimes{\mathcal O}_{\C P^1}(1)$. By a result of Kodaira \cite{Ko}
(see also the proof of Theorem 3.3. in \cite{HKLR} for more details), there exists a complex $2d$-dimensional family $\mathcal U$ of sections of $\mathcal M_{DH}$ such that for all $s\in\mathcal U$ the normal bundle is $\mathbb C^d\otimes{\mathcal O}_{\C P^1}(1)$. Hence, by the implicit function theorem, 
 we obtain a real manifold $\mathcal R\subset\mathcal U$ satisfying the properties (1) and (2). We shall see in a moment that the real dimension of $\mathcal R$ is $2d$. 
 
Recall from section \ref{sec:lambdaConn} the section
\[\omega^\C\,\in\, H^0(\mathcal M_{DH},\,\Lambda^2 T^*_F\otimes \pi^*{\mathcal O}_{\C P^1}(2)),\]
where $T_F\,=\,{\rm kernel}(d\pi)\,\subset\, T\mathcal M_{DH}$ is the tangent bundle along the fibers of $\pi$. For each $\lambda\in\mathbb CP^1$, $\omega^\C_\lambda \,=\, \omega_I^\C + \lambda\omega_I +
\lambda^2\overline{\omega_I^\C}$ is the
holomorphic symplectic form on the fiber $\pi^{-1}(\lambda)$ determined by (and determining) the hyper-K\"ahler structure. Since $\rho \,=\, \tau\circ N$ and the action of $N$ on $\pi^{-1}(0) \,=\, (\mathcal M_{SD},\,I)$ is holomorphic with respect to $I$ and anti-holomorphic with respect to $J,K$, it is easy to check that $\rho^*\omega \,=\, -\overline{\omega}$. Any section of $\mathcal M_{DH}$ induces a complex symplectic bilinear form on
$H^0(\C P^1,\,\mathcal N\otimes {\mathcal O}_{\C P^1}(-1))\,=\,\mathbb C^d$, which we again denote by $\omega^\C$. 
We equip $H^0(\mathbb CP^1,\, {\mathcal O}_{\C P^1}(1))$ with its natural symplectic structure $\eta$
given by the Wronskian, that is,
\[\eta(a+\lambda b, c+\lambda d)\,=\,ad-bc\, .\]
As in the case of the antipodal involution in \cite{HKLR} we get a complex inner product 
$g\,=\,\omega^\C\otimes\eta$
on 
\[
T_s\mathcal U \,=\, T_s\mathcal R\otimes\C\,=\,\C^d\otimes H^0(\C P^1,\,{\mathcal O}_{\C P^1}(1))\, .
\]
 
In order to describe the hyper-symplectic structure on $\mathcal R$ we proceed 
analogous to \cite{HKLR} for the hyper-K\"ahler case. The line bundles ${\mathcal O}_{\C P^1}(\pm1)\,
\longrightarrow\,\mathbb CP^1$ are real with respect to the real structure 
$\lambda\,\longmapsto\,\overline{\lambda}^{-1}$ as opposed to being quaternionic for 
the antipodal involution. As the normal bundle of a real section $s$ is naturally 
equipped with a real structure we have a real structure on
\[\mathcal N_s\,=\,\mathbb C^d\otimes {\mathcal O}_{\C P^1}(1)\]
and an induced real structure on
\[\mathcal N_s\otimes {\mathcal O}_{\C P^1}(-1)\,=\,\C P^1\times \mathbb C^d\, ,\]
which we denote by $\widehat\rho$. Then tangent vectors $X\,\in\, T_s\mathcal R$ are given 
by real sections of $\mathcal N_s$ which must be of the form
\[X\,=\,v+\lambda \widehat{\rho}(v)\]
for some $v\,\in\, \mathbb C^d.$ For $s$ corresponding to $(\nabla,\,0)$, where $\nabla$ 
is an irreducible unitary flat connection, we can describe the above data rather 
explicitly. Consider $a\,\in\,\Omega^{0,1}(\sln(E))$ and $\phi\,\in \,\Omega^{1,0}(\sln(E))$ 
such that $\partial^\nabla a \,=\, 0 \,= \,\overline{\partial}^\nabla\phi$. Then we get
a real infinitesimal deformation of the twistor line $s(\lambda)\,=\, (\overline{\partial}^\nabla, 
\,\lambda\partial^\nabla,\,\lambda)$ by considering the $t$-family
\[
 \lambda\,\longmapsto\,(\overline{\partial}^\nabla+t(a- \lambda\phi^*),
\,\lambda\partial^\nabla+t(\phi-\lambda a^*), \, \lambda),
\]
or in other words, the family
\[
 \lambda\,\longmapsto\, \nabla^\lambda(t) \,=\, \nabla + t(a-a^* +
\lambda^{-1} \phi -\lambda\phi^*)\, .
\]
The members of this family satisfy 
\[
F^{\nabla^\lambda(t)} \,=\, -t^2[a^*-\lambda^{-1}\phi,\,a-\lambda\phi^*]\, ,
\]
i.e., $\nabla^\lambda(t)$ is integrable to the first order. Thus, from $$\dot s(\lambda) \,=\, (a,\phi) + \lambda(-\phi^*,-a^*)$$
we can read off that $\widehat{\rho}(a,\phi) \,=\, -(\phi^*,a^*)$. Note that this also gives that $\dim_\R\mathcal R \,=\, 2d$. 

Now suppose that a real infinitesimal deformation vanishes at some point $\lambda_0$. Then we have 
\[
0\,=\, (a,\,\phi) -\lambda_0(\phi^*,\,a^*)\, ,
\]
which implies that $a = \lambda_0\phi^*$ and $\phi \,=\, \lambda_0a^*$. This can only hold if
$|\lambda_0|^2 \,=\, 1$. Thus, we find that the real section of the normal bundle of $s$ cannot vanish for $|\lambda|^2\neq 1$, i.e., we have established condition (3) for sections corresponding to irreducible flat unitary connections. Since this is an open condition, we may assume it to hold on the neighborhood $\mathcal R$ (after shrinking, if necessary).

The restriction of $g$ to the real subspace $T_s\mathcal R$ gives an inner product of signature $(d,d)$ as can be deduced from the
explicit form of $\omega^\C$ (see
Section \ref{sec:SDaM}) and the formula
\[
g(v+\lambda \widehat{\rho}(v), w+\lambda \widehat{\rho}(w))
\,=\,\omega^\C(v,\widehat{\rho}(w))-\omega^\C(\widehat{\rho}(v),w)\, .
\]
In fact, the subspaces given by infinitesimal deformations with $a\,=\,0$ respectively
$\phi \,=\,0$ are definite and orthogonal. Thus, the metric is in particular
non-degenerate on the submanifold corresponding to flat irreducible unitary
connections, and we may assume it to remain non-degenerate on the neighborhood
$\mathcal R$ (again possibly after shrinking).

The complex structure $I$ on $\mathcal R$ is obtained via the local identification with an open neighborhood of
the stable Higgs pairs with zero Higgs fields (which are identified with the irreducible flat unitary connections). It is just the natural complex structure on $\C^d$. It follows just like in \cite{HKLR} that $I$ is skew-symmetric with respect to $g$. 
The product structure $S$ of the hyper-symplectic structure
is given by the real structure induced by $\widehat\rho$ and the real structure on ${\mathcal O}_{\C P^1}(-1)
\,\longrightarrow\,\mathbb CP^1$ on 
\[H^0(\mathbb CP^1,\,\mathcal N_s\otimes {\mathcal O}_{\C P^1}(-1))\,=\,\mathbb C^d\]
which we identify by condition (3), as for the complex structure above, by evaluation at $\lambda=0$ with the tangent space of
$T_s\mathcal R.$ Obviously, $I$ and $S$ anti-commute.
In order to prove that $I$ and $S$ are parallel, we use the hyper-symplectic
analog of Lemma 6.8 in \cite{Hi1}, see
Lemma \ref{inthyp} below. The
closedness of the forms $\omega_P$ can be computed analogously to the hyper-K\"ahler case as in \cite{HKLR} (evaluating at $\lambda = 0$ and $\lambda = \pm \lambda_0$ for some sufficiently small $\lambda_0$ with $|\lambda_0|<1$).
\end{proof}

\begin{lemma}\label{inthyp}
Let $(M^{4k},g,I,S,T)$ be a pseudo-Riemannian manifold with skew-adjoint (with respect to $g$) endomorphisms $I,S,T\in 
\Gamma(\mathrm{End}(TM))$ satisfying $$-I^2 \,=\, S^2 \,=\, T^2 \,=\, \mathrm{Id}_{TM}\ \ {\rm and } \ \ IS \,= \,T \,=\, -SI\, .$$
Then $I,\,S,\,T$ are parallel and integrable if the 
$2$-forms $\omega_P\,=\,g(P-,-), \, P\,\in\, \{I,S,T\}$ are all closed. (The integrability condition for $S$ and $T$ is
recalled below.)
\end{lemma}

\begin{proof}
This result is well-known and its proof is very similar to the proof of the analogous statement in hyper-K\"ahler geometry. 
Suppose that $\omega_I,\omega_S,\omega_T$ are all closed. To show that $S$ is
integrable, we have to show that the $(\pm 1)$-eigen-bundles $TM^\pm\,\subset\, TM$ of
$S$ are integrable in the sense of Frobenius. Let $\epsilon\in \{1,-1\}$ and let $X,\,Y
\,\in\, \Gamma(TM^\epsilon)$. We have to show that $S[X,\,Y] \,=\, \epsilon [X,\,Y]$. Now
observe that for any $U,V\in\Gamma(TM)$ 
\[
\omega_T(U,\,V) \,=\, g(TU,\,V) \,=\, g(ISU,\,V) \,=\, \omega_I(SU,\,V)\, .
\]
This means that 
\[
SU \,= \,\epsilon U \quad \iff i_U\omega_T = \epsilon i_U\omega_I.
\]
Now, using the closeness of the $\omega_i$'s and $i_X\omega_T \,=\, \epsilon i_X\omega_I, i_Y\omega_T = \epsilon i_Y\omega_I$, we can compute
\begin{eqnarray*}
i_{[X,Y]}\omega_T &=& \mathcal L_X(i_Y\omega_T) - i_Y\mathcal L_X\omega_T \\
&=& \mathcal L_X(i_Y\omega_T) - i_Y\left(di_X\omega_T\right) \\
&=& \epsilon\mathcal L_X(i_Y\omega_I) - \epsilon i_Y\left(d(i_X\omega_I)\right) \\
&=& \epsilon\mathcal L_X(i_Y\omega_I) - \epsilon i_Y\left(\mathcal L_X\omega_I \right)\\
&=&\epsilon i_{[X,Y]}\omega_I.
\end{eqnarray*}
The integrability of $S$ follows. The integrability of $T$ follows by an analogous computation. 

To show that $I$ is integrable, we need to ensure that any $Z,\,W\,\in \,
\Gamma(T^{1,0}M)$ satisfy the equation $$I[Z,\,W] \,=\, \sqrt{-1} [Z,\,W]\, .$$
We note that an argument analogous
to the above shows that for any $U\,\in\, \Gamma(TM_\C)$, the following holds:
\[
IU\,=\, \sqrt{-1} U\,\iff\, i_{U}\omega_S \,= \,\sqrt{-1}i_{U}\omega_T\, . 
\]
Now a computation entirely analogous to the one in Lemma 6.8 of \cite{Hi} gives that
\[
i_{[Z,W]}\omega_S \,=\, \sqrt{-1} i_{[Z,W]}\omega_T\, .
\]
Thus, we have shown that $I,S,T$ are integrable. To show that $I$ is parallel we can
appeal to a standard identity in (almost) Hermitian geometry
(see \cite[\S~IX, Theorem~4.2]{K-N2} which also works for pseudo-Riemannian
metrics). The integrability of $S$ is equivalent to the vanishing of its Nijenhuis
tensor $$N_S(X,\,Y) \,=\, [X,\,Y] - S[SX,\,Y]- S[X,\,SY] + [SX,\,SY]$$ (note the
different signs in comparison to the Nijenhuis tensor of an almost complex
structure). An identity analogous to the one cited above (see \S~2 of \cite{IZ}) then shows that $S$ is parallel, and similarly $T$.
\end{proof}

\section{Automorphisms of $\mathcal M_{DH}$}\label{auto}

In recent papers \cite{Ba}, \cite{BBS}, automorphism groups of certain moduli spaces (e.g. Higgs bundle, de Rham and Betti moduli spaces) 
associated to compact Riemann surfaces have been computed. In \cite{BiHe}, the first and the second author have determined the connected 
component, containing the identity element, of the group of holomorphic automorphisms of the rank one Deligne-Hitchin moduli space. Moreover, 
it was shown in \cite{BGHL} that the Deligne-Hitchin moduli space for a compact Riemann surface $X$ is holomorphically isomorphic to the 
Deligne-Hitchin moduli space for another compact Riemann surface $Y$ if and only if $X$ is isomorphic to one of $Y$ and the conjugate Riemann 
surface $\overline{Y}$. In this final section we obtain some information about the holomorphic automorphisms group $\text{Aut}(\mathcal M_{DH})$ of 
$\mathcal M_{DH}$.

\begin{theorem}
Let $\mathcal M_{DH}$ be a Deligne-Hitchin moduli space associated to $\Sigma$.
Let $F\,\in\, {\rm Aut}(\mathcal M_{DH})$ be a holomorphic automorphism which is homotopic to the identity.
Then, $F$ maps the fibers of the projection $\pi\, :\, \mathcal M_{DH}\,\longrightarrow\,\C P^1$ to the fibers.
\end{theorem}

\begin{proof}
Assume $F$ does not map fibers to fibers, and let $\mathcal F\,\subset\,\mathcal M_{DH}$ be the image
under $F$ of a fiber $\pi^{-1}(\lambda)$ which is not contained in any fiber.
Then, for dimensional reasons, there exists $\lambda_0\,\in\,\mathbb C^*,$
and $\nabla$ a flat connection whose holomorphic and anti-holomorphic parts are stable, such that
\[p\,:=\,[\overline{\partial}^\nabla,\,\lambda_0(\partial^\nabla), \, \lambda]\in\mathcal F\]
and such that
\[d_p\lambda\,\colon\, T_p\mathcal F\,\longrightarrow\, T_{\lambda_0}\C P^1\]
is surjective. We denote by ${\mathcal F}_\lambda$ the intersection of $\mathcal F$ with the fiber over $\lambda.$
Near $\lambda_0$ we trivialize $\mathcal M_{DH}=\C^*\times\mathcal M_{dR}.$
Thus, near $\lambda_0$ and $\nabla$, $\mathcal F_\lambda$ is a complex submanifold of $\mathcal M_{dR}.$
Consider the projection $\varpi\,\,\colon\, \mathcal F\,\longrightarrow\,\mathcal M$ to the moduli space of stable holomorphic bundles,
which is well-defined (and holomorphic) near $p$ by assumption. If its differential $d_p\varpi_{\mid \text{ker} (d_p\lambda)}$ restricted to $\mathcal F_{\lambda_0}$ is not surjective,
then the restriction $d_p\widetilde{\varpi}_{\mid \text{ker} (d_p\lambda)}$ of the differential of the projection $\widetilde\varpi$ to the moduli space of anti-holomorphic structures 
would be surjective, again for dimensional reasons. Both cases can be dealt with analogously, so we assume 
that $d_p\varpi_{\mid \text{ker} (d_p\lambda)}$ is surjective.
Then, \[\varpi^{-1}(\overline{\partial}^\nabla)\,\subset\, \mathcal F\] is near $p$ a $d$-dimensional submanifold,
where $d$ is such that $\dim \mathcal M_{DH}\,=\,2d+1$, while 
 \[\varpi^{-1}(\overline{\partial}^\nabla)\cap\mathcal F_{\lambda_0}\,\subset\, \mathcal F\] is near $p$ a $(d-1)$-dimensional submanifold.
 
Therefore, there exist a $\lambda_1\,\in\,\C^*\setminus\{\lambda_0\}$ and $\Phi\,\in\, H^0(\Sigma,\,K\text{End}_0(\overline{\partial}^\nabla))$ small such that
 \[[\overline{\partial}^\nabla,\lambda_1(\partial^\nabla+\Phi), \, \lambda_1]\in\mathcal F.\]
 This implies the existence of a holomorphic section $s$ of $\mathcal M_{DH}
\,\longrightarrow\,\C P^1$
 \[s(\lambda)=[\overline{\partial}^\nabla, \lambda \partial^\nabla+\frac{\lambda_1}{\lambda_1-\lambda_0}(\lambda-\lambda_0)\Phi, \, \lambda]\]
 which goes through two different points in $\mathcal F$ (note that $s(\infty)=[0,\partial^\nabla+\Phi,0]\in\mathcal M_{Hod}(\overline{\Sigma})$
is stable as $\Phi$ is small, so that $s$ is a stable section).
We claim that $F^{-1}(s)$ cannot be homotopic to a section of $\mathcal M_{DH}
\,\longrightarrow\,\C 
P^1$. Indeed, $F^{-1}(s)$ intersects the fiber $F^{-1}(\mathcal F)$ at least twice,
so the degree of the map
\[\lambda\in\C P^1\longmapsto \pi\circ F^{-1}(s)(\lambda)\in\C P^1\]
is at least two, and this gives us a contradiction.
\end{proof}

\section*{Acknowledgements}

We are very grateful to the two referees for their very helpful comments to improve the
manuscript. The first author is supported by a J. C. Bose Fellowship. The second author
is supported by RTG 1670 ''Mathematics inspired by string theory and quantum field 
theory'' funded by the Deutsche Forschungsgemeinschaft (DFG).

\end{document}